\documentclass[12pt]{amsart}
\usepackage{a4wide,amsmath,amsbsy,amsfonts,amssymb,
stmaryrd,amsthm,mathrsfs,graphicx, amscd, tikz-cd, cancel, supertabular}
\usepackage[normalem]{ulem}
\usetikzlibrary{matrix,arrows,decorations.pathmorphing}
\usepackage{subcaption}

\usepackage{tikz}
\usetikzlibrary{matrix,arrows,decorations.pathmorphing}
\usepackage[all,cmtip]{xy}
\usepackage{qtree}

\usepackage[top=1.2in,bottom=1.2in,left=1.21in,right=1.21in]{geometry}
\usepackage[
	hypertexnames=false,
	hyperindex,
	pagebackref,
	pdftex,
	breaklinks=true,
	bookmarks=false,
	colorlinks,
	linkcolor=blue,
	citecolor=red,
	urlcolor=red,
]{hyperref}
\usepackage{hyperref}

\newcommand\bs{\backslash}

\newtheorem{theorem}{Theorem}[section]
\newtheorem{proposition}[theorem]{Proposition}
\newtheorem{corollary}[theorem]{Corollary}
\newtheorem{lemma}[theorem]{Lemma}

\theoremstyle{definition}

\newtheorem{definition}[theorem]{Definition}

\numberwithin{equation}{section}

\theoremstyle{definition}
\newtheorem{remark}[theorem]{Remark}

\newcommand{\into}{\hookrightarrow}

\newcommand{\Cc}{\mathcal{C}}
\newcommand{\Cs}{\mathscr{C}}

\newcommand{\calE}{\mathcal{E}}

\newcommand{\Os}{\mathscr{O}}
\newcommand{\Xs}{\mathscr{X}}

\newcommand{\calM}{\mathcal{M}}
\newcommand{\calB}{\mathcal{B}}

\newcommand{\calMc}{\overline{\mathcal{M}}}

\newcommand{\Tc}{\mathcal{T}}
\newcommand{\Oc}{\mathcal{O}}
\newcommand{\Pc}{\mathcal{P}}

\newcommand{\Cb}{\mathbb{C}}

\newcommand{\C}{\mathbb{C}}

\newcommand{\F}{\mathbb{F}}

\newcommand{\Pb}{\mathbb{P}}

\newcommand{\Q}{\mathbb{Q}}
\newcommand{\Rb}{\mathbb{R}}

\newcommand{\Z}{\mathbb{Z}}

\newcommand{\sgn}{\textrm{sgn}}

\newcommand{\Sf}{\mathfrak{S}}
\newcommand{\Af}{\mathfrak{A}}
\newcommand{\Cf}{\mathfrak{C}}
\newcommand{\Df}{\mathfrak{D}}

\newcommand{\one}{\mathbf{1}}

\newcommand{\Aut}{\rm Aut}

\newcommand{\beq}{\begin{eqnarray}}
\newcommand{\eeq}{\end{eqnarray}}

\newcommand{\SL}{\textrm{SL}}
\newcommand\ssm{\smallsetminus}

\newcommand{\Lcal}{\mathcal{L}}

\newcommand{\ir}{{ir}}

\DeclareMathOperator{\Pic}{Pic}
\DeclareMathOperator{\PGL}{PGL}
\DeclareMathOperator{\GL}{GL}

\DeclareMathOperator{\Sym}{Sym}

\DeclareMathOperator{\ev}{ev}

\DeclareMathOperator{\PU}{PU}
\DeclareMathOperator{\SU}{SU}
\DeclareMathOperator{\SO}{SO}

\DeclareMathOperator{\la}{\langle}
\DeclareMathOperator{\ra}{\rangle}

\DeclareMathOperator{\Sing}{Sing}
\DeclareMathOperator{\odd}{odd}

\title[Geometry of the Wiman-Edge pencil, I]{Geometry of the Wiman-Edge pencil, I:  \\
algebro-geometric aspects\\
\bigskip
{\em In memory of W.L. Edge}}
\author{Igor Dolgachev}
\address{Department of Mathematics, University of Michigan} 
\author{Benson Farb}
\address{Department of Mathematics, University of Chicago} 
\author{Eduard Looijenga}
\address{Yau Mathematical Sciences Center, Tsinghua University, Beijing}

\begin{document}
\maketitle

\begin{abstract}
In 1981 W.L. Edge \cite{Edge1,Edge2} discovered and studied a pencil $\Cs$ of highly symmetric genus $6$ projective curves with remarkable properties.   Edge's work was based on an 1895 paper \cite{Wiman1} of A. Wiman.  Both papers were written in the satisfying style of 19th century algebraic geometry.  In this paper and its sequel \cite{FL}, we consider $\Cs$ from a more modern, conceptual perspective, whereby explicit equations are reincarnated as geometric objects.  
\end{abstract}
\tableofcontents

\section{Introduction}

The purpose of this paper is to explore what we call the {\em Wiman-Edge  pencil} $\Cs$, a pencil of highly symmetric genus $6$ projective curves with remarkable properties.  The smooth members of $\Cs$ can be characterized (see Theorem \ref{thm:onaquinticDelPezzo}) as those smooth genus $6$ curves admitting a nontrivial action of alternating group $\Af_5$.   This mathematical gem was discovered 
in 1981 by  W.L. Edge \cite{Edge1,Edge2},  who based his work on a discovery in 1895 by A. Wiman of a non-hyperelliptic curve $C_0$ of genus 6 (now called the {\em Wiman curve}) with automorphism group isomorphic to the symmetric group $\Sf_5$. The results of both Wiman and Edge are in the satisfying style of 19th century mathematics, with results in terms of explicit equations.  

In this paper and its sequel \cite{FL}, we consider $\Cs$ from a more modern, conceptual perspective, whereby explicit equations are reincarnated as geometric objects.  In particular we will view $\Cs$ through a variety of lenses: from algebraic geometry to representation theory to hyperbolic and conformal geometry.  We will address both modular and arithmetic aspects of $\Cs$, partly through Hodge theory.  Given the richness and variety 
of structures supported by $\Cs$, we can say in hindsight that the Wiman-Edge pencil deserved such a treatment, and indeed it seems odd that this had not happened previously.

In the introduction of his {\em Lectures on the Icosahedron} \cite{Klein}, Klein writes:  ``A special difficulty, which presented itself in the execution of my plan, lay in the great variety of mathematical  methods  entering in the theory of the Icosahedron.''  We believe that this is still very much true 
today, in ways Klein probably never anticipated. This, by the way,  is followed by the sentence: ``On this account it seemed advisable to take granted no specific knowledge in any direction, but rather to introduce,  where necessary, such explanations $\ldots$''  While we are writing only about one aspect of Klein's book, in this paper we have tried to take Klein's advice seriously.

\subsection*{The Wiman-Edge pencil} As mentioned above, the story starts with A. Wiman \cite{Wiman1}, who, while classifying algebraic curves of genus $g = 4,5$ and $6$ whose automorphisms group contains a simple group,  discovered  a curve $C_0$ of genus 6 with automorphism group isomorphic to the symmetric group $\Sf_5$. On the last page of his paper, Wiman gives the equation of a $4$-nodal plane sextic birationally isomorphic to $W$ :
$$2\sum x^4yz+2\sum x^3y^3-2\sum x^4y^2+\sum x^3y^2z+\cdots)-6x^2y^2z^2 = 0,$$
 He reproduces this equation on p. 208 of his later paper \cite{Wiman}, related to the classification of finite subgroups of the plane Cremona group.  Wiman states there that the group of birational automorphisms of $C_0$ is generated by a group of projective transformations isomorphic to the symmetric group $\Sf_4$ together with the standard quadratic birational involution with fundamental points at its three singular points.   
 
Wiman erroneously claims that his curve is the unique non-hyperelliptic curve of genus 6 whose automorphism group contains a non-cyclic simple group, the alternating group $\Af_5$ in his case. This mistake was corrected almost a hundred years later by W. Edge \cite{Edge1,Edge2}, who placed $W$ inside a pencil 
\[\lambda P(x,y,z)+\mu Q(x,y,z) = 0\] of  four-nodal plane sextics each of whose  members  admits a group of automorphisms isomorphic to $\Af_5$. In projective coordinates different from the one chosen by Wiman, the pencil is generated by the curves defined by the homogeneous equations :

\[P(x,y,z)=(x^2-y^2)(y^2-z^2)(z^2-x^2) = 0\]
and 
\[Q(x,y,z)=x^6+y^6+z^6+(x^2+y^2+z^2)(x^4+y^4+z^4)-12x^2y^2z^2 = 0,\]
where the Wiman curve $W$ corresponds to the parameters $(\lambda:\mu)= (0:1)$. 

Edge showed that it is more natural to view the pencil as a pencil $\Cs$ of curves on a quintic del Pezzo surface $S$ obtained by blowing up the four base points of the pencil.  In this picture, the lift $C_0$ to $S$ of Wiman's curve $W$ is, as Edge discovered, ``a uniquely special canonical curve of genus $6$'' on $S$.    That is, the standard action of $\Sf_5$ on the quintic del Pezzo surface $S$ permutes the $1$-parameter family of smooth genus $6$ curves on $S$ but leaves invariant a unique such curve, namely the Wiman sextic $C_0$.  

The action of $\Sf_5$ on $S$ induces an action of $\Sf_5$ on $\Cs$, whereby the subgroup $\Af_5$ leaves each member of the pencil $\Cs$ invariant, and acts faithfully on each curve by automorphisms.  In addition to $C_0$, the pencil $\Cs$ has precisely one other $\Sf_5$-fixed member, a reducible curve that is a union of $10$ lines intersecting in the pattern of the Petersen graph (see Figure \ref{peterseng}), with the $\Sf_5$ action on this union inducing the standard $\Sf_5$ action on the Petersen graph.  The other four singular members of $\Cs$ come in pairs, the curves in each pair being switched by any odd permutation.   A schematic of $\Cs$, and the $\Sf_5$ action on it, is given (with explanation) in Figure \ref{figure:schematic1}.  

\begin{figure}
\centerline{\qquad\qquad\includegraphics[scale=0.4]{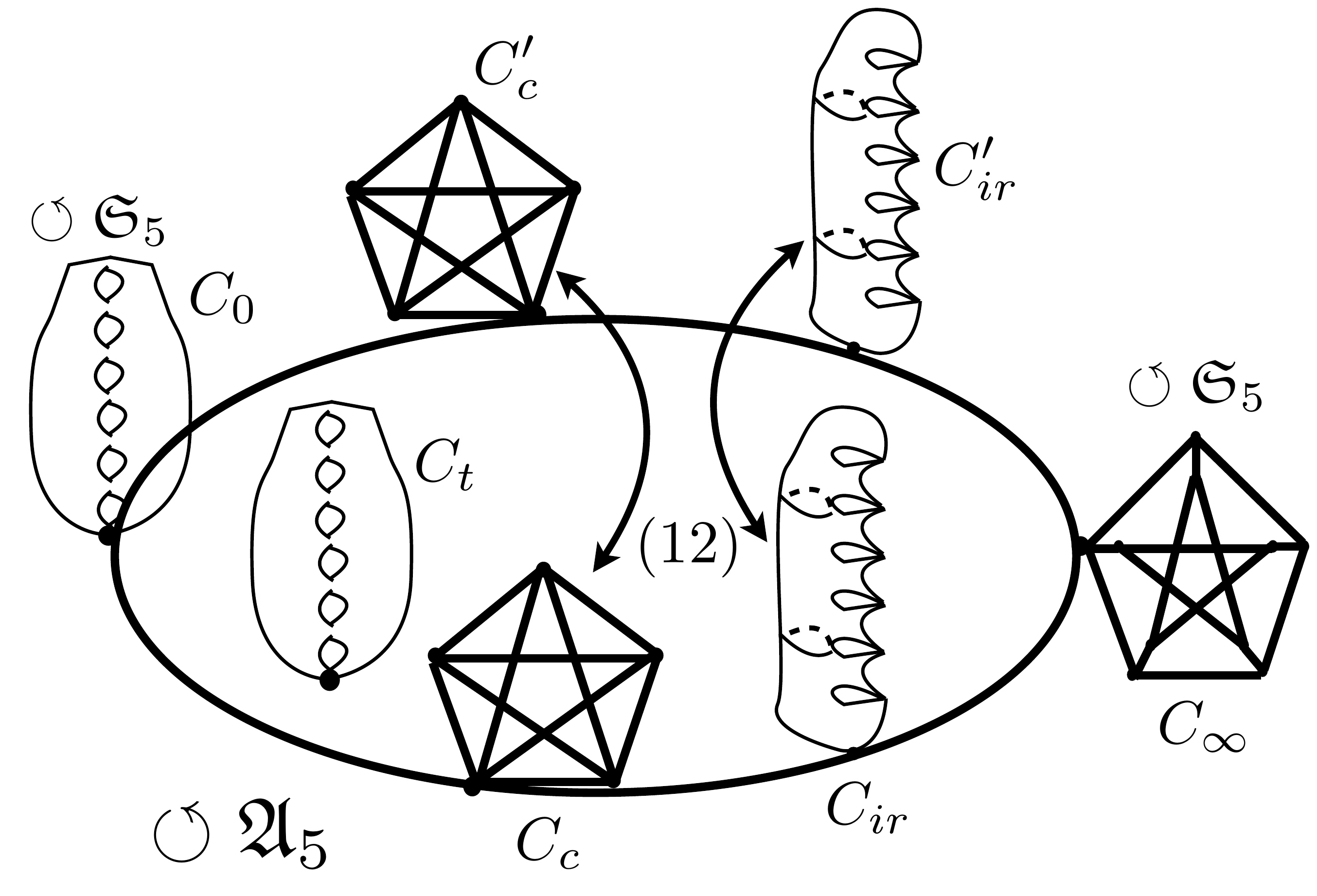}}
\caption{\footnotesize 
A schematic of the Wiman-Edge pencil $\Cs$, consisting of genus $6$ projective curves containing the alternating group $\Af_5$ in their automorphism group.  The generic curve $C_t\in \Cs$ has $\Aut(C_t)\cong\Af_5$, while there is a unique smooth member $C_0\in\Cs$, the {\em Wiman sextic}, with more symmetry: $\Aut(C_0)\cong \Sf_5$.  There are $5$ singular members of $\Cs$: two irreducible curves, $6$-noded rational curves $C_{ir}$ and $C'_{ir}$; two curves $C_c$ and $C'_c$, each consisting of $5$ conics whose intersection graph is the complete graph on $5$ vertices; and a union $C_\infty$ of $10$ lines whose intersection graph is the Petersen graph. The group $\Sf_5$ acts on $\Cs$ with $\Af_5$ leaving each member of $\Cs$ invariant.  This action has two $\Sf_5$-invariant members, $C_0$ and $C_\infty$, while any odd permutation switches $C_{ir}$ with $C'_{ir}$ and $C_c$ with $C'_c$.}
\label{figure:schematic1}
\end{figure}

Since Edge's discovery, the pencil $\Cs$, which we propose to call the \emph{Wiman-Edge pencil},  has appeared in several modern research papers. For example, its nonsingular members have been identified with the quotient of one of the two $1$-parameter families of lines on a nonsingular member of the Dwork pencil of Calabi-Yau quintic threefolds:
$$x_1^5+x_2^5+x_3^5+x_4^5+x_5^5+5\phi x_1x_2x_3x_4x_5 = 0$$ 
by the group of order 125 of obvious symmetries of the Dwork pencil (see \cite{Candelas}, \cite{Zagier}; we elaborate  a bit on the connection in the present paper).The Wiman-Edge occupies a prominent place in monograph \cite{Cheltsov} on birational geometry of algebraic varieties with $\Af_5$-symmetry. It also appears in a recent posting by  Cheltsov-Kuznetsov-Shramov \cite{CKS}.

\subsection*{This paper}The purpose of the present paper is two-fold. First, we reprove all previously known facts about the Wiman-Edge pencil that can be found in Edge's papers. Instead of computations based on the pencil's equation (used also in a later paper \cite{Gonzalez} whose authors it seems were not  aware of Edges's work), our proofs rely on the representation theory of the symmetry group of the pencil, and also on the moduli-theoretical interpretation of the quintic  del Pezzo  surface. Although  our approach is not new, and has been used also in \cite{Cheltsov} or \cite{InoueKato}, respectively, we would like to 
believe that our methods are more conceptual and geometric.

The other goal of the paper is to answer some natural questions that arise while one gets familiar with the pencil. Thus we give a purely geometric proof of the uniqueness of the Wiman-Edge pencil as a $\Af_5$-invariant family of stable curves of genus 6, and in particular, the uniqueness of the Wiman curve $W$ as a non-hyperelliptic curve of genus 6 with the group of automorphisms isomorphic to $\Sf_5$. In fact, we give two different proofs of this result: an algebraic geometrical one given in 
Theorem   \ref{thm:onaquinticDelPezzo} below, and a second one in the sequel \cite{FL} 
that is essentially group-theoretical and topological. 

Secondly, we give a lot of attention to  the problem only barely mentioned by Edge and later addressed in paper \cite{Slodowy}: describe a $\Af_5$-equivariant projection of the Wiman-Edge pencil to a Klein plane realizing an irreducible $2$-dimensional projective representation of $\Af_5$. It reveals a natural relation between the Wiman-Edge pencil and the symmetry of the icosahedron along the lines of  Klein's book \cite{Klein}. In particular, we relate the singular members of the pencil with some attributes of the geometry of the Clebsch diagonal cubic surface as well with some  rational plane curves of degree 6 and 10 invariant with respect to a projective group of automorphisms isomorphic to $\Af_5$ which
 were discovered by R. Winger \cite{Winger}.

In the sequel \cite{FL} we will discuss some other aspects of the Wiman paper related to hyperbolic geometry,  the moduli space of curves and Shimura curves.

\subsection*{Section-by-section outline of this paper.} Section \ref{sect:quinticDP} collects a number of mostly known facts  regarding quintic del Pezzo surfaces, but with emphasis on naturality, so that it is straightforward to keep track of how the automorphism group of such a surface acts on the vector spaces associated to it. We also recall the incarnation of such a surface as the Deligne-Knudsen-Mumford compactification $\calMc_{0,5}$ and mention  how some its features can be recognized in either description.
Perhaps new is Lemma \ref{lemma:plueckeremb} and its use to obtain Proposition \ref{prop:quadricsintersection}. 

Section \ref{sect: wimansextic}  introduces the principal object of this paper, the Wiman-Edge pencil $\Cs$. Our main result here is that its smooth fibres define the universal family of genus $6$ curves endowed with a faithful $\Af_5$-action. We also determine the singular fibers of $\Cs$. Each of these turns out to be a stable curve. Indeed, the whole pencil consists of all the stable genus $6$ curves endowed with a faithful $\Af_5$-action that can be smoothed as a $\Af_5$-curve.
We further prepare for the next two sections by describing $\Sf_5$-orbits in $\calMc_{0,5}$, thus recovering a list due to Coble. This simplifies considerably when we restrict to the $\Af_5$-action and this what we will only need here.

The following two sections concern the projection to a Klein plane. These are $\Af_5$-equivariant  projections as mentioned above.
Section \ref{sect:deg5cover} concentrates on the global properties of this projection, and proves among other things that the ramification curve of such a projection is in 
fact a singular member of the Wiman-Edge pencil. It is perhaps worthwhile to point out that we prove this using the Thom-Boardman formula for cusp singularities.  This is the only instance we are aware of where such a  formula for a second order Thom-Boardman symbol is used to prove an algebro-geometric property. We also show that the image of the Wiman-Edge pencil in a Klein plane  under the  projection is no longer a pencil, but a curve of degree 5 with two singular points. 

As its title makes clear,  Section \ref{sect:WEimages} focusses on the images  
in the Klein plane of special members of the Wiman-Edge pencil. We thus find ourselves suddenly staring at a gallery of planar representations of degree $10$ of 
genus $6$ curves with $\Af_5$-action, probably  all known to our predecessors in the early 20th century if not earlier. Among them stand out (what 
we have called) the \emph{Klein decimic}  and the \emph{Winger decimic}. Less exciting  perhaps at first is the case of a conic with multiplicity $5$; but this turns 
out to be the image of a member of the Wiman-Edge pencil which together with its $\Sf_5$-conjugate is characterized by possessing  a pencil of even theta characteristics.

In the final section, Section \ref{sect:orbitspace}, we look at the $\Sf_5$-orbit space of  $\calMc_{0,5}$ (which is just the Hilbert-Mumford compactification 
of the space of  binary quintics given up to projective equivalence) and make the connection with the associated invariant theory.

\subsection*{Acknowledgements. } We thank Shigeru Mukai for helpful information regarding genus $6$ curves.  

\subsection*{Conventions.} 
Throughout this paper the base field is $\C$ and $S$ stands for a quintic del Pezzo surface (the definition is recalled in Section \ref{sect:quinticDP}).
The canonical line bundle $\Omega_M^n$ of a complex $n$-manifold $M$ will be often denoted by $\omega_M$. As this is also the dualizing sheaf for $M$, we use the same notation if $M$ is possibly singular, but has such a sheaf (we only use this for curves with nodes).

For a vector space $V$, let $\Pb(V)$ denote the projective space of 1-dimensional subspaces of $V$ and $\check\Pb (V)=\Pb (V^\vee)$ denotes the projective space of hyperplanes of $V$.
We write $\Sym^dV$ for $d$th symmetric power of the vector space $V$. For a space or variety $X$ we denote by $\Sym^dX$  the quotient of $X^d$ by the permutation group $\Sf_d$.

\section{Del Pezzo surfaces of degree $5$}\label{sect:quinticDP}

\subsection{A brief review}\label{subsect:generalities}
Here we recall some known facts about quintic del Pezzo surfaces, i.e., del Pezzo surfaces of degree 5,  that one can find in many sources (such  as \cite{CAG}).  By definition  a \emph{del Pezzo surface} 
is a smooth projective algebraic surface $S$ with ample anticanonical bundle $\omega_S^{-1}$. For such a surface $S$, the first Chern class map 
$\Pic(S)\to H^2(S; \Z)$  is an isomorphism. Denoting by  $K_S\in H^2(S; \Z)$ the canonical class of $S$, i.e.,  the  first Chern class of $\omega_S$ (and hence minus the first Chern class of the tangent bundle of $S$), the self-intersection number $d = (-K_S)^2 = K_S^2$ is called the \emph{degree} of $S$. With the exception of surfaces  isomorphic to $\Pb^1\times \Pb^1$ (which are del Pezzo surfaces of degree 8), a del Pezzo surface of degree $d$ admits a birational morphism $\pi:S\to \Pb^2$ whose inverse is the blow-up of $9-d$ distinct points (so we always have $d\le 9$) satisfying some genericity conditions. But beware that when $d\le 6$, there is more that one way to contract curves in $S$ that produce a copy of $\Pb^2$.

When $d=5$ (which we assume from now on),  these genericity conditions amount to having no three (of the four) points lie on a line so that we can
adapt our coordinates  on $\Pb^2$ such that the points in question, after having them numbered $(p_1,p_2,p_3, p_4)$,  are the vertices of the coordinate system: \[p_1=(1:0:0), \ p_2=(0:1:0), \ p_3=(0:0:1), \ p_4=(1:1:1).\] 
This shows that any two quintic del Pezzo surfaces are isomorphic and also proves  that  the automorphism group of $S$ contains the permutation group $\Sf_4$. The complete linear system $|-K_S|:= \Pb(H^0(S,\omega_S^{-1}))$ defined by nonzero sections of the dual of the canonical line bundle of $S$ is then the strict transform of the linear system in $\Pb^2$ of plane cubic curves passing through these points in the sense that a general member is the strict transform of such a cubic. This is a linear system of dimension $9-4=5$ and gives an embedding in a $5$-dimensional projective space
\[
S\hookrightarrow \Pb_S := \check{\Pb}(H^0(S,\omega_S^{-1}))
\]
with image a surface of degree $5$.  
We call $S$, thus  embedded  in a projective space, an \emph{anticanonical model} of $S$. 

The automorphism group of $S$ is, however, bigger than $\Sf_4$: the embedding $\Sf_4\hookrightarrow \Aut (S)$ extends to an isomorphism 
\[\Sf_5\cong \Aut (S)\]
given by assigning to the transposition $(45)$ the lift of the Cremona transformation 
\[(t_0:t_1:t_2)\mapsto (t_0^{-1}:t_1^{-1}:t_2^{-1}),\] 
which indeed lifts to an involution of $S$ (see  \cite[Theorem 8.5.8]{CAG}).  

The image of a nonsingular rational curve on $S$ with self-intersection number $-1$ (in other words, an {\em exceptional curve of the first kind}) is a line on the anticanonical model, and any such line is so obtained. Any line on the quintic del Pezzo surface $S$ is  either the strict transform of a line through $p_i$ and $p_j$ ($1\le i<j\le 4$) or the preimage of $p_i$ ($i=1,\dots, 4$), so that there are $\binom{4}{2}+4=10$ lines on $S$. 

The  strict transform $\Pc_i$ of the pencil of lines through $p_i$ makes a pencil $\Pc_i$ of rational curves  on $S$ whose members are 
pairwise disjoint: they are the fibers of a morphism from $S$ to a projective line. The strict transform of the conics through $p_1,\dots, p_4$ make up a pencil $\Pc_5$ with the same property. The intersection number of a member of $\Pc_i$ with 
$-K_S$ is equal to the intersection number of a cubic  with a  line  in $\Pb^2$ minus 1 if $i\ne 5$ (resp. cubic with a conic minus $4$ if $i = 5$), hence is equal to $2$. This is therefore also a conic in the anticanonical model, and that is why we refer to $\Pc_i$ as a \emph{pencil of conics} and call the corresponding morphism from $S$ to a projective line   a \emph{conic bundle}.

Every pencil of conics has exactly three singular fibers; these are unions of two different lines. For example, the pencil $\Pc_i, i\ne 5$, has singular members equal to the pre-images on $S$ of the lines $\overline{p_ip_j}$ joining the point $p_i$ with the point $p_j, j\ne i$. The singular members of pencil $\Pc_5$ are proper transforms of the pairs of lines $\overline{p_ip_j}+\overline{p_kp_l}$ with all indices distinct. 

It follows that each of ten lines on $S$ is realized as an irreducible component of exactly three different pencils of conics. This allows one to label it with a $2$-element subset of $\{1, \dots, 5\}$, being the complement of the corresponding $3$-element subset.
We thus obtain a bijection between the set of 10 lines and the collection of $2$-element subsets of  $\{1, \dots, 5\}$, with two lines intersecting if and only if the associated $2$-element subsets are disjoint, or equivalently, when the two lines constitute a singular member of one of the five pencils of conics. 
Thus the {\em intersection graph} $\mathcal{G}$ of the set lines, i.e., the graph whose  vertex set is the set of lines and in which two vertices are connected by an edge if and only if the associated lines intersect,  is the Petersen graph (Figure \ref{peterseng}). We see here $\Aut (S)$ also represented as the  automorphism  group of $\mathcal{G}$.

\begin{figure}[ht]
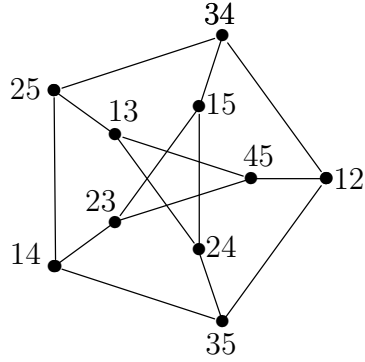

\centering
\xy (-80,0)*{};(20,0)*{\bullet};(6.2,19)*{\bullet}**\dir{-};(-16.2, 11.7)*{\bullet}**\dir{-};(-16.2, 11.7)*{\bullet};(-16,-11.7)*{\bullet}**\dir{-};
(-16.2,-11.7)*{\bullet};(6.2,-19)*{\bullet}**\dir{-};(6.2,-19)*{};(20,0)*{\bullet}**\dir{-};
(10,0)*{\bullet};(-8.1,5.85)*{\bullet}**\dir{-};(-8.1,5.85)*{};(3.1, -9.5)*{\bullet}**\dir{-};(3.1, -9.5)*{};(3.1,9.5)*{\bullet}**\dir{-};
(3.1,9.5)*{};(-8.1,-5.85)*{\bullet}**\dir{-};
(-8.5,-5.85)*{};(10,0)*{}**\dir{-};   (10,0)*{};(20,0)*{}**\dir{-};  (3.1,9.5)*{};(6.2,19)*{}**\dir{-};
 (-16.2, 11.7)*{};(-8.1,5.85)*{}**\dir{-};(-16.2, -11.7)*{};(-8.2,-5.85)*{}**\dir{-};(6.2, -19)*{};(3.1,-9.85)*{}**\dir{-};
 (23,0)*{12};(5.8,22)*{34};(5.8,22)*{34};(-20,12)*{25};(6,-22)*{35};(11,3)*{45};(6,10)*{15};(6,-9)*{24};(-7,9)*{13};
 (-10,-3)*{23};(-20,-10)*{14};
\endxy
\caption{\small The Petersen graph $\mathcal{G}$.  This is the graph with vertices corresponding to $2$-element subsets of $\{1,2,3,4,5\}$, with two vertices $\{i,j\}$ and $\{k,\ell\}$ connected by an edge 
precisely when $\{i,j\}\cap \{k,\ell\}=\emptyset$. The graph $\mathcal{G}$ is the intersection graph of the $10$ lines on the quintic del Pezzo surface, with a vertex for each line and an edge connecting vertices corresponding to lines with nontrivial intersection.}\label{peterseng}
\end{figure}

The edges of the Petersen graph represent reducible conics on $S$, and  there are indeed $3\times 5 = 15$ of them.
The $3$ reducible conics of a conic  bundle are then represented by three disjoint edges whose indices are three $2$-element subsets of a subset 
of $\{1,2,3,4,5\}$ of cardinality $4$. If we regard the missing item  as a label for the conic pencil, we  see that  $\Aut(S)\cong\Sf_5$ is realized as the full 
permutation group of the set of conic pencils.

\subsection{The modular incarnation.}\label{subsect:conicfibrations} 
A quintic del Pezzo surface has the---for us very useful---incarnation as 
the Deligne-Mumford moduli space of stable 5-pointed rational curves, 
$\calMc_{0,5}$, see \cite{Kapranov}. In the present case, this is also the 
Hilbert-Mumford (or GIT) compactification of $\calM_{0,5}$, which means that a point of $\calMc_{0,5}$ is uniquely represented up to automorphism by a \emph{smooth} rational curve $C$ and  a $5$-tuple $(x_1,\dots , x_5)\in C^5$ for which the divisor
$\sum_{i=1}^5 (x_i)$ has all its multiplicities $\le 2$. One of the advantages of either model is that the $\Sf_5$-action is evident.
The ten lines appear as the irreducible components of the boundary $\partial\calMc_{0,5}=\calMc_{0,5}\ssm \calM_{0,5}$ of the compactification:
these are the  loci defined by $x_i=x_j$, where $\{i,j\}$ is a $2$-element subset of $\{1, \dots, 5\}$.
We thus recover the bijection between the set of 10 lines and the collection of  2-element subsets of $\{1, \dots, 5\}$.

The conic bundles also have a modular interpretation, namely 
as forgetful maps: if a point of  $\calMc_{0,5}$ is represented by a Deligne-Mumford stable curve $(C; x_1, \dots , x_5)$, then  forgetting 
$x_i$ ($i=1,\dots, 5$), followed by contraction  unstable components  (and renumbering the points by $\{1,2,3,4\}$ in an order preserving manner) yields an element of $\calMc_{0,4}$. There is a  similar definition in 
case a point of $\calMc_{0,5}$  is represented in a Hilbert-Mumford stable  manner,  but beware that a Hilbert-Mumford stable 
representative of  $\calMc_{0,4}$ is given uniquely up to isomorphism by a $4$-tuple $(x_1, \dots , x_4)\in C^4$  for which the divisor $\sum_{i=1}^4(x_i)$ is either reduced or twice a reduced divisor; if for example $x_1=x_2$, 
then we also require that $x_3=x_4$ (but $x_2\not= x_3$). The moduli space $\calMc_{0,4}$ is a smooth rational curve and 
$\partial\calMc_{0,4}= \calMc_{0,4}\ssm \calM_{0,4}$ consists of three points indexed by the three ways we can partition $\{1, 2,3, 4\}$ into two $2$-element subsets. 
The resulting morphism  $f_i: \calMc_{0,5}\to \calMc_{0,4}$ represents $\Pc_i$ and over the three points of $\partial\calMc_{0,4}$ lie the three singular fibres. 

Let us focus for a moment on $\Pc_5$, in other words, on  $f_5: \calMc_{0,5}\to \calMc_{0,4}$. Each $x_i$, $i=1, \dots, 4$, then defines a section 
$\calMc_{0,4}\to \calMc_{0,5}$ of $f_5$, and $f_5$ endowed with these four sections  can be understood as the ``universal stable 4-pointed rational curve''. Note that the images of these sections are irreducible components of  $\partial\calMc_{0,5}$  and hence lines in the anticanonical model. They are indexed 
by the unordered pairs $\{i, 5\}$, $i=1,\dots, 4$. The singular fibers
of $f_5$ are those over the $3$-element set $\partial\calMc_{0,4}$; each such fiber is a union of two intersection lines, so that in this way all $10$ lines are accounted for.

The $\Sf_5$-stabilizer  of the conic bundle  defined by $f_5$ is clearly $\Sf_4$. Let us take a closer look at how $\Sf_4$ acts on $f_5$. First observe that every smooth
fiber of $f_5$ (i.e., a strict transform of a smooth conic) meets every fiber of every $f_i$  ($i\not=5$) with multiplicity one.
We also note that  $\Sf_4$ acts on the $3$-element set $\partial\calMc_{0,4}$  as its full permutation group, the kernel being the Klein Vierergruppe. So the Vierergruppe acts trivially on $\calMc_{0,4}$,  but its action on the universal pointed rational curve $\calMc_{0,5}$ is of course faithful. 
The homomorphism $\Sf_4\to \Sf_3$ can also be understood as follows.  If we are given a smooth  rational curve $C$ endowed and a four element subset  $X\subset C$, then the double cover $\tilde C$ of $C$ ramified at  $X$ is a smooth genus one curve. 
If we enumerate the points of $X$ by $X=\{ x_1, x_2, x_3, x_4\}$, then  we can choose $x_1$ as origin, so that $(\tilde C, x_1)$ becomes an elliptic curve.  This then makes $\{ x_1, x_2, x_3\}$ the set of elements of $(\tilde C,x_1)$ of order $2$. Thus $\calM_{0,4}$ can be regarded as the moduli space  of elliptic curves endowed with a principal level $2$ structure. This also suggests that we should think of $X$ as an affine plane over $\F_2$; for this interpretation the three boundary points of $\calMc_{0,4}$ are cusps, and correspond to the three directions in this plane. We now can think of $\Sf_4$ as the 
affine group of $X$, the Vierergruppe as its translation subgroup, and the quotient $\Sf_3$ as the projective linear group $\PGL_2(\F_2)$.

Choose an affine coordinate $z$ for the base of $\calMc_{0,4}$ such that $\partial\calMc_{0,4}$ is the root set of $z^3=1$. Then the full permutation  group of $\partial\calMc_{0,4}$ is the group of M\"obius transformations  consisting of the  rotations $z\mapsto \zeta z$ and  the involutions  $z\mapsto \zeta/z$, with $\zeta^3=1$. From this we see that this action has 3 irregular orbits: two $3$-element orbits: besides the root set of $z^3=1$ which is $\partial \calMc_{0,4}$,  the root set of $z^3=-1$, and one $2$-element orbit: $\{ 0,\infty\}$. The root set of $z^3=-1$ is represented by $\Pb^1$ with an enumeration of the $4$th roots of unity in $\C\subset \Pb^1$,  and the third orbit is given by $\{0,\infty\}$. 

We will be more concerned with the last orbit that gives rise to two \emph{special conics} in the pencil. So every conic bundle has two special conics as fibers and $\Sf_5$ permutes these $10$ special conics  transitively. Two distinct special conics intersect 
unless they are in the same pencil. The two points $0,\infty$ have the same $\Sf_3$-stabilizer in $\Pb^1$ of order $3$. In the action of $\Sf_4$ on the total space of the pencil, the two special conics are fixed by the subgroup $\Af_4$. It follows from the preceding that these two fibers define the locus represented by  
$(x_1, \dots, x_5)\in C^5$ (with $C$ a smooth rational curve) for which there exists an affine coordinate $z$ on $C$ such that 
$\{z(x_1), \dots, z(x_4)\}$ is the union of $\{\infty\}$ and the root set of $z^3=1$. Our discussion also shows that the subgroup $\Af_5$ has two orbits in that set, with each orbit having exactly one special conic in each conic bundle.

\subsection{Representation spaces of $\Sf_5$}\label{subsect:repspaces}
In what follows we make repeated use of the representation theory of $\Sf_5$ and $\Af_5$, and so let us agree on its notation. 
In the Tables \ref{table:S5} and \ref{table:A5}, the columns are the conjugacy classes of the group, indicated by the choice of a representative.
We  partially followed Fulton-Harris \cite{FultonHarris} for the notation of the types of the irreducible representations of $\Sf_5$. 
Here $\one$ denotes the trivial representation, $\sgn$ denotes the sign representation, $V$ denotes the Coxeter representation, that is, the standard $4$-dimensional irreducible representation of $\Sf_5$ and $E$ stands for $\wedge^2V$.   Note that each of these representations is real and hence admits  a nondegenerate $\Sf_5$-invariant quadratic form, making it self-dual.

Our labeling of the irreducible representations of  $\Af_5$  overlaps with that of $\Sf_5$, and this is deliberately so: the $\Sf_5$-representations $V$ and $V\otimes\sgn$ become isomorphic when restricted to $\Af_5$,  but  remain irreducible and that is why we  still denote them by $V$. The same applies to $W$ and $W\otimes\sgn$ (and of course to $\one$ and $\sgn$). On the other hand, the restriction to $\Af_5$ of $\Sf_5$-representation $E:=\wedge^2\,V$ is no longer irreducible, but is 
isomorphic as an $\Af_5$-representation $I\oplus I'$ (cf.\ Table \ref{table:A5}). The representations $I$ and $I'$ differ by the outer automorphism of $\Af_5$ induced by conjugation with an element of  $\Sf_5\ssm \Af_5$. Both $I$ and $I'$ are  realized as the group of isometries of Euclidean $3$-space that preserve an icosahedron. In particular, they are real (and hence orthogonal).

We will often use the fact that the natural map $\Sf_5\to \Aut(\Af_5)$ (given by conjugation)  is an isomorphism of groups. So the outer automorphism group of $\Af_5$ is of order two 
and representable by conjugation with an odd permutation.

\tablecaption{The character table of $\Sf_5$.}\label{table:S5}
\begin{center}
\begin{supertabular}{|r||c|c|c|c|c|c|c |}
\hline
type & $1$ & $(12)$ & $(12)(34)$ & $(123)$ & $(123)(45)$ & $(1234)$ & $(12345)$\\
\hline\hline
$\one$ & $1$ & $1$ & $1$ & $1$ & $1$ & $1$ & $1$ \\
\hline
$\sgn$ &$1$ & $-1$ & $1$ & $1$ & $-1$ & $-1$ & $1$ \\
\hline
$V$  & $4$ & $2$ & $0$ & $1$ & $-1$ & $0$ & $-1$ \\
\hline
$V\otimes \sgn$  & $4$ & $-2$ & $0$ & $1$ & $1$ & $0$ & $-1$ \\
\hline
$W$  & $5$ & $1$ & $1$ & $-1$ & $1$ & $-1$ & $0$\\
\hline
$W\otimes \sgn$  & $5$ & $-1$ & $1$ & $-1$ & $-1$ & $1$ & $0$ \\
\hline
$E:=\wedge^2\,V$ & $6$ & $0$ & $-2$ & $0$ & $0$ & $0$ & $1$ \\ 
\hline
\end{supertabular}
\end{center}
\vskip5mm
\tablecaption{The character table of $\Af_5$.}\label{table:A5}
\begin{center}
\begin{supertabular}{|r||c|c|c|c|c|c|c |}
\hline
type & $(1)$ & $(12)(34)$ & $(123)$ & $(12345)$ & $(12354)$\\
\hline\hline
$\one$ & $1$ & $1$ & $1$ & $1$ & $1$\\
\hline
$V$  & $4$ & $0$ & $1$ & $-1$ & $-1$\\
\hline
$W$  & $5$ & $1$ & $-1$ & $0$ & $0$\\
\hline
$I$ & $3$ & $-1$ & $0$ & $(1+\sqrt{5})/2$ & $(1-\sqrt{5})/2$\\ 
\hline
$I'$ & $3$ & $-1$ & $0$ & $(1-\sqrt{5})/2$ & $(1+\sqrt{5})/2$\\ 
\hline
\end{supertabular}
\end{center}
\vskip5mm

The isomorphism $\Aut(S)\cong \Sf_5$ depends of course on the model of $S$ as a blown-up $\Pb^2$, but since  $\Aut(S)$ permutes these models transitively, this isomorphism is unique up to an inner automorphism. This implies that the characters of $\Aut(S)$, or equivalently, the isomorphism types of the finite dimensional irreducible representations of this group, are naturally identified with those of $\Sf_5$.

Via this action of $\Sf_5$ on $S$, all linear spaces naturally associated to $S$ can be made into linear representations of $\Sf_5$. 
For example, $H^2(S; \C)$ contains the trivial representation, spanned by the anticanonical class $c_1(S)=-K_S$.   The intersection pairing gives an integral, symmetric bilinear form on $H^2(S;\Z)$.  
The orthogonal complement of $K_S$ in $H^2(S; \Z)$, denoted here by $H^2_0(S; \Z)$, contains (and is spanned by) a root system of type $A_4$, the roots being the elements of 
self-intersection $-2$. A root basis is \[(e_1-e_2,e_2-e_3,e_3-e_4,e_0-e_1-e_2-e_3)\] where $e_0$ is the class of a preimage of a line in $\Pb^2$ and $e_i$ is the class of the  line over $p_i$. This identifies $H^2_0(S; \C)$ as a $\Sf_5$-representation with the Coxeter representation $V$. The (classes of) lines themselves in $H^2(S; \Z)$ make up a $\Sf_5$-orbit.  In the language of  root systems, it is the orbit of a fundamental weight; for the given root basis  this weight is represented by the orthogonal projection of $e_4$ in $H^2_0(S; \Q)$. Thus $H^2_0(S; \C)$ is isomorphic to $V$ as a $\Sf_5$-representation. 

Another  example is $H^0(S,\omega_S^{-1})$, a $6$-dimensional representation of $\Sf_5$. Using the 
explicit action of $\Sf_5$ on $S$ and hence on the space of cubic polynomials representing elements of $H^0(S,\omega_S^{-1})$, we find that it has the same character as $E=\wedge^2V$.  This is why we shall write $E_S$ for $H^0(S,\omega_S^{-1})$, 
so that  $|-K_S|=\Pb (E_S)$ and $\Pb_S=\check\Pb (E_S)$.

\subsection{The Pl\"ucker embedding}\label{subsect:pluecker}  We here show how $S$ can be obtained  in an intrinsic manner as a linear section  of  the Grassmannian 
of lines in projective $4$-space. The naturality will make this automatically  $\Aut(S)$-equivariant. 
Let $C_\infty$ denote the union of the ten lines on $S$. It is clear that $C_\infty$ is a normal crossing divisor in $S$ (under the 
modular interpretation of $S\cong \calMc_{0,5}$ it is the Deligne-Mumford boundary).

\begin{lemma}\label{lemma:plueckeremb}
The sheaf of logarithmic differentials $\Omega_S^1(\log C_\infty)$ is globally generated.  The action of $\Sf_5$ on $S$ makes the space of global sections $W_S:=H^0(S;\Omega_S^1(\log C_\infty))$ an irreducible $5$-dimensional $\Sf_5$-representation isomorphic to $W$.
\end{lemma}

\begin{proof} Denote by $\Lcal$ the $10$-element set of lines on $S$ and consider the \emph{residue exact sequence}
\beq\label{residue}
0\to \Omega_S^1\to \Omega_S^1(\log C_\infty)\overset{\textrm{res}}{\longrightarrow} \oplus_{L\in\Lcal}\, i_{L*}\Oc_{L}\to 0,
\eeq
where $i_L:L\hookrightarrow S$ is the inclusion.
All homomorphisms here are natural and the exact sequence of cohomology gives an exact sequence of $\Sf_5$-representations: 
\[
0\to H^0(S;\Omega_S^1(\log C_\infty))\to \C^{\Lcal} \overset{\phi}{\to} H^1(S;\Omega_S^1).
\]
The $\Sf_5$-representation $\C^{\Lcal}$ is the permutation representation on the vertices of the Petersen graph. 
Its character is that of $\one\oplus V\oplus W$, as one can see by looking at pairs of complementary pentagons inside of the Petersen graph. 
We also note that  $H^1(S;\Omega_S^1)\cong H^2(S;\C)$ is as a  representation isomorphic to $\one\oplus V$. 
The blow-up model makes it obvious that  $H^2(S;\C)$ is generated  by the 
classes of lines, so that the homomorphism $\phi$ is surjective, and we obtain an isomorphism of representations 
\[
W_S=H^0(S;\Omega_S^1(\log C_\infty)) \cong W.
\]
It remains to show that $W_S$ generates  $\Omega_S^1(\log C_\infty)$.
For this, we use the  description of $S$ as the blowup of $\Pb^2$ in the vertices of its coordinate simplex. In terms of the 
coordinates $(t_0:t_1:t_2)$, the logarithmic forms $dt_1/t_1-dt_0/t_0$ and $dt_2/t_2-dt_0/t_0$ generate the sheaf of differentials on $\Pb^2$ outside the coordinate simplex  and so they define elements $\zeta_1$  and $\zeta_2$ of $W_S$ which generate $\Omega_{S\ssm C_\infty}$. Since  $\Aut(S)$ is transitive on the vertices of the Petersen graph, it remains to prove generation along just one line, say the strict transform of $t_1=0$. A straightforward computation shows  that $\zeta_1$ and $\zeta_2$ take care of this, 
except at the point at infinity (where $t_0=0$), but  the origin is included. Since $\Aut(S)$ is also transitive on the edges of the Petersen graph,  this finishes the proof.
\end{proof}

\begin{remark}\label{rem:}
A local computation shows right away that $\det \Omega_S^1(\log C_\infty)= \omega_S(C_\infty)$. Since $C_\infty$ is a divisor of $\omega_S^{-2}$, the latter is isomorphic to $\omega_S^{-1}$, so that $c_1(\Omega_S^1(\log C_\infty))=-K_S$. The residue exact sequence allows one to compute $c_2(\Omega_S^1(\log C_\infty))$ also:   we have
\[
c(\Omega_S^1(\log C_\infty)) = 1+c_1+c_2 = \frac{c(\Omega_S^1)}{\prod_{L\in \Lcal}(1-[L])}= (1+K_S+c_2(S))\prod_{L\in \Lcal}(1+[L]+[L]^2).
\]
As is well-known, $c_2(S)$ is the cohomology  class in top degree whose value on the fundamental class is the Euler characteristic of $S$, which is $7$
(since $S$ is the  blow up of four points in $\Pb^2$). 
  Using the identities $L^2 = -1$ and $\sum_{L\in \Lcal} [L] = -2K_S$, together with the fact that the Petersen graph has 15 edges, we then find that
$\la c_2(\Omega_S^1(\log C_\infty)), [S]\ra =2$.
\end{remark}

Since $\Omega_S^1(\log C_\infty)$ is globally generated, the evaluation map $e_x:W_S\to \Omega_S^1(\log C_\infty)(x))$ is onto for every 
$x\in S$. So $\ker  (e_s)$ is of dimension $3$ and the annihilator of this kernel in $W_S^\vee$  is of dimension $2$. We thus obtain a morphism
\[
f:S\to G(2,W^\vee_S) = G_1 (\Pb(W_S^\vee))= G_1(\check\Pb(W_S))
\]
Under this map, the pull-back of the dual of the tautological subbundle on the Grassmannian 
becomes isomorphic to $\Omega_S^1(\log C_\infty)$. 

\begin{proposition}\label{prop:quadricsintersection} 
The morphism $f$ is a closed, $\Sf_5$-equivariant embedding. Its composition with the Pl\"ucker embedding 
\[
G_1(\Pb(W_S^\vee))\hookrightarrow \check\Pb(\wedge^2\,W_S)\cong \Pb^9
\] 
is the anticanonical embedding into the subspace $\check\Pb(E_S)$, where $E_S$ is identified with a direct summand of the representation 
\beq\label{lambdasquare}
\wedge^2\, W_S \cong E_S\oplus (V\otimes\sgn).
\eeq
In particular, the anticanonical model of $S$ is 
$\Sf_5$-equivariantly isomorphic to the intersection of $G_1(\Pb(W^\vee_S))$ with the subspace $\Pb(E)$ in its Pl\"ucker embedding.
\end{proposition}

\begin{proof}  A  locally free sheaf $\mathcal{E}$ of rank $r$ on a compact variety $X$ that is globally generated determines a morphism  $X\to G(r,H^0(\calE)^\vee)$ whose composite with the Pl\"ucker embedding 
$G(r,H^0(\calE)^\vee)\hookrightarrow \Pb(H^0(\calE)^\vee)$ is given by the  invertible sheaf  $\det \mathcal{E}$. Applying this to our situation, we find that the composite of $f$ with the Pl\"ucker embedding is given by the complete linear system $|-K_S|$. We know that it defines a closed embedding, from this the first claim follows.

To see the second claim, we compute the character of the representation of $\Sf_5$ in $\wedge^2\, W_S$ by  means of the formula 
$\chi_{\wedge^2\,W}(g) = \frac{1}{2}(\chi_W(g)^2-\chi_W(g^2))$. The standard character theory gives us the decomposition \eqref{lambdasquare}. 
\end{proof}

\begin{remark} The linear map $F:V\otimes\sgn\to \wedge^2\,W^\vee$ of $\Sf_5$-representations found above defines an equivariant linear embedding \[\Pb^4 \cong \Pb(V)=\Pb (V\otimes\sgn)\hookrightarrow \check\Pb(\wedge^2\,W)\cong \Pb^9\] whose
geometry is discussed in Example 10.2.20 of \cite{CAG}. We note that this give rise to a rational map 
\[
\tilde{F}:\Pb(V)\dasharrow \Pb(W)
\] 
obtained as the composite of the Veronese map $\Pb (V)\to \Pb(\Sym^2 V)$ and the projectivization of the linear map 
\[\Sym^2 V\xrightarrow{\wedge^2F} \Sym^2  (\wedge^2\,W^\vee)\to \wedge^4 W^\vee\cong \det (W^\vee)\otimes W.\]
It is clear that $\tilde{F}$ is given by a 4-dimensional linear system of quadrics. Its  indeterminacy locus  is where $F (v)$ has rank
$\le 2$ instead of $4$ (as one might expect) and  consists of 5 points in general position (so these are the vertices of a coordinate simplex of the 4-dimensional $\Pb(V)$) and $\tilde{F}$ is then given by the linear system of quadrics through these.  
The $\Sf_5$-equivariant map $\tilde{F}$ appears in modular setting which we shall now recall (and which is discussed for example in \S 9.4.4 of \cite{CAG}).

The moduli space $\calMc_{0,6}$ of stable 6-pointed genus zero curves is isomorphic to the blow-up of $\Pb^3$ at the 5  vertices of its coordinate simplex  followed by the blow-up of the proper transforms of the lines of this simplex (see \cite{Kapranov}). There is a natural map from $\calMc_{0,6}$ to the Hilbert-Mumford compactification of $\calM_{0,6}$. The latter 
appears as the image an  $\calMc_{0,6}$ in a $4$-dimensional projective space via  the linear system of quadrics in $\Pb^3$ through the 5 coordinate vertices so that this reproduces a copy of our map $\Phi$ above.  The image  of this map is  a cubic  hypersurface,
called  the \emph{Segre cubic} $\mathcal{S}_3$. 

The Segre cubic has 10 nodal points, each of which is the image of an exceptional divisor over a  line of the coordinate simplex. Observe that the modular interpretation of the morphism $\calMc_{0,6}\to\mathcal{S}_3$
makes evident an action of $\Sf_6$, although in the above model, only its restriction to the subgroup $\Sf_5$  is manifest. The ambient 4-dimensional projective space of the Segre cubic $\mathcal{S}_3$ is the projectivization of an irreducible  $5$-dimensional representation of $\Sf_6$ corresponding to the partition  $(3,3)$ (see \cite{CAG}, p. 470). 
\end{remark}

\subsection{The anticanonical model}\label{subsect:acmodel}
 Some of what follows can be found in  Shepherd-Barron \cite{SB}; see also Mukai \cite{mukai}. 

We have just seen that an  anticanonical model  of $S$ is obtained as  linear section of a Grassmannian $G_1(W_S^\vee)$ for its Pl\"ucker embedding. 
It is well-known that the Pl\"ucker equations define the image of  $G_1(W_S^\vee)\hookrightarrow \Pb (\wedge^2 W^\vee_S)=\check\Pb(\wedge^2 W_S)$ as the intersection of 
five quadrics: these are given by  the map 
\[
\begin{array}{lll}
\wedge^2 W^\vee_S&\to& \wedge^4 W^\vee_S=\det(W_S^\vee)\otimes W_S\\
\alpha&\mapsto&\alpha\wedge\alpha
\end{array}\]
or rather by its dual $\det(W_S)\otimes W_S^\vee\to \Sym^2(\wedge^2 W_S)$. The latter is a nonzero map of $\Sf_5$-representations. 
As we mentioned,  the irreducible representation  $W$ is self-dual, and the character table of $\Sf_5$ shows that the $(-1)$-eigenspace of the transposition   $(12)\in\Sf_5$  in $W$ has dimension $2$. Hence $\det(W_S)$ is the trivial representation and $\det(W_S)\otimes W^\vee_S$ can as a $\Sf_5$-representation be identified with the irreducible representation $W_S$. Thus we have obtained an $\Sf_5$-equivariant embedding 
\[W_S\hookrightarrow \Sym^2(\wedge^2 W_S)\]
 (unique up to scalars) and  this realizes $\Pb(W_S)$ as the linear system of  (Pl\"ucker) quadrics in  
$\check\Pb (\wedge^2 W_S)$ that define $G_1(W_S^\vee)$. By restriction to  $E_S^\vee$, we can also understand $W_S$ as defining the linear system of quadrics in  $\check\Pb (E_S)$ that has  $S=G_1(W_S^\vee)\cap \check\Pb(E_S)$ as its base locus. When thus interpreted, we will write $I_S(2)$ for $W_S$.

The character of $\Sym^2E_S$ can be computed by means of the formula $\chi_{\Sym^2E}(g) = \frac{1}{2}(\chi_E(g^2)+\chi_E(g)^2)$. We then find that
\beq\label{symdec}
\Sym^2E_S \cong W^{\oplus 2}\oplus (W\otimes \sgn)\oplus V\oplus \one\oplus \sgn.
\eeq
We have already singled out the subrepresentation $I_S(2)$ as a copy of $W$ and so the remaining summands in \eqref{symdec} will add up to a representation that embeds in  $H^0(S,\omega_S^{-2})$. An application of Riemann-Roch or the explicit discussion below shows that this is in fact an isomorphism, so that we have an exact sequence of $\Sf_5$-representations
\[
0\to I_S(2)\to S^2H^0(S, \omega_S^{-1})\to H^0(S, \omega_S^{-2})\to 0.
\]
and  a $\Sf_5$-equivariant isomorphism 
\beq\label{anti-bicanonical}
|\omega_S^{-2}| \cong \check{\Pb} (W\oplus (W\otimes\sgn)\oplus V\oplus\one\oplus \sgn) \cong \Pb^{15}.
\eeq
The summands $\one$  and $\sgn$ in $\Sym^2E_S$ are explained by restricting the representation $E_S$ to $\Af_5$: then $E_S$ decomposes into
two irreducible $\Af_5$-representations: $E_S=I\oplus I'$, each of which is orthogonal. This means that $(\Sym^2 I)^{\Af_5}$ and   $(\Sym^2 I')^{\Af_5}$ are of dimension one. If $Q$ is a generator of $(\Sym^2 I)^{\Af_5}$, then its image $Q'$ under an element of $\Sf_5\ssm \Af_5$ is a generator of $(\Sym^2 I')^{\Af_5}$. So $Q+Q'\in \Sym^2E_S$ is $\Sf_5$-invariant (it spans the $\one$-summand) and $Q-Q'\in \Sym^2E_S$ transforms according to the sign character (it spans the $\sgn$-summand). The above discussion shows that their  images in  $H^0(S, \omega_S^{-2})$ remain independent.

\begin{corollary}\label{cor:Cinfty}
The image of $Q-Q'$ in $H^0(S, \omega_S^{-2})$, which spans a copy of the sign representation,  has divisor $C_\infty$.
\end{corollary}
\begin{proof}
As is well-known, `taking the  residue at infinity' identifies  the space of rational 3-forms on $\C^3$ that are invariant under scalar multiplication (i.e.,  are homogeneous of degree zero) with the space of rational 2-forms on $\Pb^2$. Thus
\[
\alpha:= \frac{(dt_0\wedge dt_1\wedge dt_2)^2}{t_0t_1t_2(t_0-t_1)(t_1-t_2)(t_2-t_0)}
\]
can be understood as  a rational section of  $\omega_{\Pb^2}^{2}$ whose divisor is minus the sum of the six coordinate lines. It follows that $\pi^*\alpha$ can be regarded as a generating section  of  $\omega_S^{2}(C_\infty)$. Hence its inverse $\pi^*\alpha^{-1}$ becomes a generating section  of
$\omega_S^{-2}(-C_\infty)$, making $H^0(S, \omega_S^{-2}(-C_\infty))$ a $\Sf_5$-invariant subspace of $H^0(S, \omega_S^{-2})$. Exchanging $t_1$ and $t_2$ clearly takes $\alpha$ to $-\alpha$ and
so the span $H^0(S, \omega_S^{-2}(-C_\infty))$ of  $\pi^*\alpha^{-1}$ is a copy of the sign representation in $H^0(S, \omega_S^{-2})$. It follows that $\pi^*\alpha^{-1}$ must be proportional to the image  of $Q-Q'$ in $H^0(S, \omega_S^{-2})$, so that the divisor of $Q-Q'$ is $C_\infty$.
\end{proof}

\begin{remark}\label{rem:explicitquadrics}
We can make the canonical embedding and  $I_2(S)$ explicit in terms of the blowup model of $S$.
Let us first observe that the Petersen graph has twelve $5$-cycles (pentagons). The stabilizer subgroup of a $5$-cycle is the dihedral subgroup $D_{10}$ of $\Sf_5$ of order $10$. For example,  the cycle formed by the vertices $(12),(23),(34),(45),(15)$ is stabilized by the subgroup generated by permutations  $(12345)$ and $(25)(34)$. A geometric interpretation of a $5$-cycle is a hyperplane section of $S$ which consists of a pentagon of lines. Note that they  come in pairs with respect to taking the complementary subgraph. A pentagon of lines, viewed as  reduced divisor on $S$, is a member of the anticanonical system of $S$. For our description of $S$ as a blown-up $\Pb^2$ these must be transforms  of triangles in $\Pb^2$. The list of these triangles is as follows:

\begin{align*}
f_0 &= t_1t_2(t_0-t_2),& f_0' &= t_0(t_0-t_1)(t_1-t_2);\\
f_1 &= t_1(t_0-t_1)(t_0-t_2),&  f_1' &= t_0t_2(t_1-t_2);\\
f_2 &=(t_0-t_1)(t_0-t_2)t_2, & f_2' &= t_0t_1(t_1-t_2);\\
f_3 &= t_1t_2(t_0-t_1), & f_3' &= t_0(t_0-t_2)(t_1-t_2);\\
f_4 &= t_1(t_0-t_2)(t_1-t_2), &  f_4' &= t_0t_2(t_0-t_1);\\
f_5 &= t_2(t_1-t_2)(t_0-t_1), & f_5'&= t_0t_1(t_0-t_2).
\end{align*} 

A direct check gives that the left-hand column can be linearly expressed in terms of the 
right-hand column (and vice versa), as follows:

\begin{align*}
f_0' =& f_1-f_2+f_5,  &f_1'= &f_0-f_3+f_5,  & f_2' = f_0-f_3+f_4,\\
f_3'= & f_1-f_2+f_4,  &f_4' =&f_2+f_3-f_5,  & f_5' = f_0+f_1+f_4.
\end{align*} 

But beware that in order to get actual of sections of $\omega_S^{-1}$, we need to multiply these elements with $(dt_0\wedge dt_1\wedge dt_2)^{-1}$:
the resulting $3$-vector fields are then invariant under scalar multiplication and have a residue at infinity that can be understood as 
an element of $H^0(S, \omega_S^{-1})$
Note that $(f_0,\ldots,f_5)$ is a basis of the linear space $H^0(S,\omega_S^{-1})$, and so these basic elements 
can serve as the coordinates  of an anticanonical embedding $S\into \Pb^5$.
Since $f_if'_i = f_jf_j'$, we see that $S$ is contained in the intersection of quadrics defined by equations 
$x_ix_i'-x_jx_j' = 0$, where $x_i' $ is the linear form in $x_0, \dots, x_5$ that expresses 
$f'_i$ in terms of  the $f_0, \dots , f_5$. The five linear independent quadratic forms
$x_ix_i'-x_0x_0' $,  ($i = 1,\ldots,5$) then give us the defining equations for $S$ in $\Pb^5$.
We easily check that the quadratic form  $\sum_{i=1}^6x_ix_i'$  corresponding to the sum of pentagons of lines spans the $\sgn$ summand. It cuts out on $S$ the union $C_\infty$ of lines on $S$. The linear space spanned by quadratic forms $\{x_ix_i'\}_{i=1}^6$  decomposes as the direct sum $\sgn\oplus W$. As we have already observed, the space of quadrics $V(x_ix_i'-x_6x_6')$ span the kernel $I_S(2)$ of   
 $\Sym^2H^0(S, \omega_S^{-1})\to H^0(S, \omega_S^{-2})$. 

It is a priori clear that  $\Sf_5$ permutes the pentagons of lines listed above, but we see that it in fact preserves 
the 12-element set  and the preceding discussion makes explicit how.
\end{remark}

\section{The Wiman-Edge pencil and its modular interpretation}\label{sect: wimansextic}

\subsection{The Wiman-Edge pencil and its base locus}\label{subsect:pencil}
We found in Subsection \ref{subsect:acmodel} that there are exactly two $\Sf_5$-invariant quadrics on $\check{\Pb} (E_S)$, one defined by  $Q+Q'$  and spanning the $\one$-summand, the other by $Q-Q'$  and spanning the $\sgn$-summand and  (by Corollary \ref {cor:Cinfty}) cutting out on $S$ the $10$-line union ${C_\infty}$.   The trivial summand spanned by $Q+Q'$ cuts out a  curve that we shall call the {\em Wiman curve} and denote by ${C_0}$. We shall find it to be smooth of genus 6. We observe that the plane spanned by $Q$ and $Q'$ is the fixed point set of $\Af_5$ in $\Sym^2E_S$ and hence defines 
a pencil $\Cs$ of curves on $S$ whose members come with a  (faithful) $\Af_5$-action. This pencil is of course spanned by $C_0$ and $C_\infty$ and these are the only members that are in fact $\Sf_5$-invariant.  We refer to $\Cs$ as the \emph{Wiman-Edge pencil}; we sometimes also use this term for its image 
 in $\Pb^2$ under the natural map $\pi:S\to\Pb^2$. 

\begin{lemma}[{\bf Base locus}]\label{lemma:basepoints}
The base locus $\Delta:=C_0\cap C_\infty$ of $\Cc$ is   the unique 
20-element $\Sf_5$-orbit in $C_\infty$. The curves $C_0$ and $C_\infty$  intersect transversally so that each member  of the 
Wiman-Edge pencil is smooth at $\Delta$. 
\end{lemma}
\begin{proof}
Since this $C_0\cap C_\infty$ is $\Sf_5$-invariant, it suffices to determine how a line $L$ on $S$ meets $C_0$. We first note that 
 the intersection number $C_0\cdot L$ (taken on $S$) is $-2K_S\cdot [L]=2$. 
When we regard $L$ as an irreducible component of $C_\infty$, or rather, as defining a vertex of the Petersen graph, then
we see that the other lines meet $L$ in 3 distinct points and that the $\Sf_5$-stabilizer of $L$ acts on $L$ as the full permutation group of these three points. So if we choose an affine coordinate $z$ on $L$ such that the three points in question are the third roots of $1$, then 
we find that  the $\Sf_5$-stabilizer of $L$ acts on $L$ with three irregular orbits: two of size 3 (the roots of $z^3-1$  and the roots of $z^3+1$) and one of size 2 ($\{z=0, z=\infty\}$). It follows that the $C_0$ meets $L$ in the size 2 orbit.  In particular, the intersection of $C_0$ with $C_\infty$ is transversal and contained in the smooth part of $C_\infty$.
\end{proof}

\subsection{Genus 6 curves with $\Af_5$-symmetry}\label{subsect:wimansextic}
Any reduced $C\in \Cs$, being a member of $|-2K_S|$, has its normal sheaf in $S$  isomorphic to $\Oc_C\otimes \omega_S^{-2}$.
By the adjunction formula, its dualizing sheaf  $\omega_C$ is therefore isomorphic to 
\[\Oc_C\otimes \omega_S^{-2}\otimes_{\Oc_S}\omega_S=
\Oc_C\otimes \omega_S^{-1}=\Oc_C(1).\] 
In particular
\[\deg (\omega_C)=\deg (\Oc_C\otimes \omega_S^{-1})=(-2K_S)\cdot(-K_S)=10.\] 
It follows that $C$ has arithmetic genus $6$, that the natural map $E_S\to H^0(C, \omega_C)$ is a  $\Af_5$-equivariant isomorphism and that $C$ is canonically embedded in $\Pb_S$. 
Our goal is to give  the Wiman-Edge pencil a modular interpretation. 

\begin{proposition}[{\bf $\Af_5$ and $\Sf_5$ orbit spaces}]
\label{prop:orbifolds} 
Let  $C$ be a smooth projective curve genus $6$ endowed with a faithful action of $\Af_5$. 
Then $\Af_5\backslash C$ is of genus zero and $\Af_5$ has $4$ irregular orbits with isotropy orders $3$, $2$, $2$ and $2$.

If the $\Af_5$-action extends to a faithful $\Sf_5$-action, then  
$\Sf_5\backslash C$ is of genus zero and $\Sf_5$ has $3$ irregular orbits with isotropy orders $6$, $4$ and $2$.
The union of these irregular $\Sf_5$-orbits is also the union of the irregular $\Af_5$-orbits: the $\Sf_5$-orbit with isotropy order $6$ resp.\ $4$ is
a $\Af_5$-orbit with isotropy order $3$ resp.\ $2$  and the $\Sf_5$-orbit with isotropy order $2$ decomposes into two $\Af_5$-orbits with the same isotropy groups. 
(In other words, the double cover $\Af_5\backslash C\to \Sf_5\backslash C$ only ramifies over the points of ramification of order $6$ and $4$.)
Such a curve exists.
\end{proposition}

See also Theorem 5.1.5 of \cite{Cheltsov}.

\begin{proof}
The stabilizer of a point of a smooth curve under  a faithful finite group action is cyclic. So when  $G=\Sf_5$,  the stabilizer of a point of $C$  is a cyclic group of  order $i\le 6$, but when $G=\Af_5$, it cannot of order $4$ or $6$. Let for $i>1$, $k_i$ be the number of $G$-orbits of  size $|G|/i$. By the Hurwitz formula we have 
\[10/|G|=2(g'-1)+\sum_{i=2}^6 \frac{i-1}{i}k_i,\] where $g'$ is the genus of the quotient. 

For $G=\Sf_5$ we have  $|G| =120$ and so this shows right away that $g'=0$. It follows that 
\[25=6k_2+8k_3+ 9k_4+ \frac{48}{5} k_5+10k_6.\] 
It is then immediate that $k_5=0$ and $k_6\le 1$.
We find that  the only solutions for $(k_2,k_3,k_4, k_6)$ are $(1,0,1,1)$   and $(0,2,1,0)$. The latter possibility clearly does not occur because it would imply that an element of order 6 acts without fixed points, contradicting the Hurwitz equality $10= 6(2g'-2)$. 
 
For $G=\Af_5$, we have  $|G| =60$ and  hence  we then also have $g'=0$. The formula now becomes 
becomes $13=3k_2+4k_3+\frac{24}{5}k_5$, which has  as only solution $(k_2,k_3, k_5)=(3,1,0)$.  

The assertion concerning the map $\Af_5\backslash C\to \Sf_5\backslash C$ formally  follows from the above computations.

The last assertion will follow from Riemann existence theorem, once we find  a regular $\Sf_5$-covering of  $\Pb^1\ssm \{0,1,\infty\}$ with the simple loops yielding monodromies $\alpha,\beta ,\gamma\in\Sf_5$ of order  $6,4,2$ respectively, which generate $\Sf_5$ and for which $\alpha\beta\gamma=1$.
This can be arranged: take $\alpha=(123)(45)$ and $\beta=(1245)$ and $\gamma=(14)(23)$.
\end{proof}

We next show that any smooth projective curve genus $6$ endowed with a faithful action of $\Af_5$ appears in the Wiman-Edge pencil. For this 
we shall invoke a theorem of
S. Mukai \cite{mukai}, which states that a canonical smooth projective curve $C$ of genus 6 lies on a quintic del Pezzo surface if and only if it is neither bielliptic (i.e., it does not  double cover  a genus 1 curve), nor trigonal  (it does not triple cover a genus zero curve), nor  isomorphic to a plane quintic.

\begin{theorem}\label{thm:onaquinticDelPezzo}
Every  smooth projective curve of genus 6 endowed with a faithful $\Af_5$-action is $\Af_5$-equivariantly isomorphic to a member of the Wiman-Edge pencil.
This member is unique up to the natural action of the involution $\Sf_5/\Af_5$.
\end{theorem}
\begin{proof}
Let $C$ be such a curve. We first show that $C$ is not hyperelliptic, so that we have a canonical model.  
If it were,  then it is so in a unique manner so that the set of its 14 Weierstra\ss\ points is in invariant with respect to $\Af_5$. 
But we found in Proposition \ref{prop:orbifolds} that $\Af_5$ has in $C$ one irregular orbit of size $20$,  three of size $30$, and no others, and thus such an invariant subset cannot exist. 

From now on we assume that $C$ is canonical. We first show that it is neither trigonal, nor bielliptic, nor isomorphic to a plane quintic.

\medskip
\noindent
\emph{$C$ is trigonal: } This means that $C$ admits a base point free pencil of degree $3$. This pencil is then unique (\cite{ACGH1}, p.\ 209)  
so that the $\Af_5$-action on $C$ permutes the fibers. Consider the associated morphism $C\to \Pb^1$. The Riemann-Hurwitz formula then tells us that the ramification divisor of this morphism on $C$ has degree $16$. It must be $\Af_5$-invariant. But our list of orbit sizes precludes this possibility and  so such a divisor cannot exist.

\medskip
\noindent
\emph{$C$ is isomorphic to a plane quintic: }  It is then so in a unique manner (\cite{ACGH1}, p.\ 209)  and hence the $\Af_5$-action on $C$ will extend as a projective representation to the ambient $\Pb^2$. 

The resulting projective representation cannot be reducible, for then $\Af_5$ has a fixed point, $p\in \Pb^2$ say,  and   the action of 
$\Af_5$ on the tangent space $T_p\Pb^2$ will be faithful. But as Table \ref{table:A5} shows,  $\Af_5$ has no faithful representation. So the projective representation is  irreducible and hence the projectivization of copy of  $I$ or $I'$. Either representation is orthogonal and so the ambient 
projective plane contains a  $\Af_5$-invariant conic. The quintic defines  on this conic an effective divisor of degree $10$. Since $\Af_5$ acts on the conic (a Riemann sphere) as the group of motions of a regular  icosahedron, no $\Af_5$-orbit on this conic has fewer than $12$ points and so such a divisor cannot exist.

\medskip
\noindent
\emph{$C$ is bielliptic: } This means that $C$ comes with an involution $\iota$ whose orbit space is of genus one.   Let $G\subset \Aut(C)$ be the subgroup generated by $\Af_5$ and $\iota$. By a theorem of Hurwitz, $|\Aut(C)|\le 84(6-1) = 420$ and so $[G:\Af_5]$ can be $1$, $2$, $4$ or $6$.  

Let us first deal with the index $6$ case. For this we note that the $G$-action (by left translations) on $G/\Af_5$ has a kernel contained in $\Af_5$. This kernel is a normal subgroup and contained in $\Af_5$. Since $\Af_5$ is simple, this kernel is either trivial or all of $\Af_5$. It cannot be  all of $\Af_5$, because $G/\Af_5$ is then cyclic of order $6$ and hence cannot be generated by the image of $\iota$. It follows that $G$ acts faithfully on $G/\Af_5$
so that we get an embedding of $G$ in $\Sf_6$. Its image is then a subgroup of index $2$ and so this image must be $\Af_6$: $G\cong\Af_6$.  We now invoke the Hurwitz  formula to this group action: the stabilizer of a point is cyclic of order $\le 6$ and so if  for a divisor $i$ of $|G|$, 
$k_i$ is the number of $G$-orbits in $C$ of  size $|G|/i$, then
\[
\textstyle 10/360=2(g'-1)+\sum_{i=2}^6 \frac{i-1}{i}k_i,
\] 
where $g'$ is the genus of the quotient.  This implies that $g'=0$ and then this comes down to
\[
 73=18k_2+ 24k_3+27k_4+ \tfrac{144}{5}k_5+30 k_6.
\]
It is clear that $k_5$ must be zero. Since the left hand side is odd, we must have $k_4=1$, and we then  find  that no solution exists.

If the index $\le 4$, then the argument above gives a map $G\to \Sf_4$. Its kernel is contained in $\Af_5$, but cannot be trivial for reasons of cardinality. It follows that the kernel equals $\Af_5$, so that $\Af_5$ is normal in $G$. Since the image of $\iota$ generates $G/\Af_5$, it then follows that the index is $1$ or $2$. 

In the last case, $G$ is either isomorphic to $\Af_5\times\Cf_2$ or to
$\Sf_5$, depending on whether or not conjugation by $\iota$ induces in $\Af_5$ an inner automorphism.  In the first case, the action of 
$\Af_5$ descends to a faithful action on the elliptic curve. This would make $\Af_5$ an extension of finite cyclic group by an abelian group, which is evidently not the case. It follows that $G\subset \Sf_5$ and that $\iota$ is conjugate to $(12)$  or to $(12)(34)$. 

Denote by $\chi$ the character of the $\Sf_5$-representation $H^0(C, \Omega_C)$. Since the quotient of $C$ by $\iota$ has genus 1, we have $\chi(\iota)=-5+1=-4$.
If $\iota$ is conjugate to $(12)$, then  we then read off from Table \ref{table:S5} that $\chi$ is the character of $\mathbf{1}\oplus\sgn^{\oplus 5}$ or $\sgn^{\oplus 2} \oplus (V\otimes\sgn)$. Its restriction to $\Af_5$ is then the  trivial character resp.\  $\mathbf{1}^{\oplus 2}\oplus V$. But this contradicts the fact that the $\Af_5$-orbit space of $C$ has genus zero. If $\iota$ is conjugate to $(12)(34)$, then we then read off from Table \ref{table:A5}  that the $\Af_5$-representation $H^0(C, \Omega_C)$ takes on $\iota$ a value  $\ge -2$, which contradicts the fact that this value equals $-4$.

\smallskip
According to Mukai \cite{mukai}, it now follows that $C$ lies on a weak quintic Del Pezzo surface $S_C$ in $\check\Pb (H^0(C, \omega_C))$.  It may have singular points, and in fact quadric sections of a weak del Pezzo quintic form a divisor $D$ in the moduli space ${\mathcal M}_6$.  However, we claim that $C$ must lie on a smooth Del Pezzo surface.  

To prove this claim, first note that if $C$ it lies on a singular surface then it has fewer than five $g_4^1$'s, where by a $g_4^1$ we mean a linear series of degree $4$ and dimension $1$.  In the plane model this is because three points are collinear or two points coincide. The five $g_4^1$Õs are defined by four pencils through nodes of the sextic and the pencil of conics through $4$ nodes. In the singular case, when, we choose $3$ collinear points, there is no pencil of conics.  The divisor 
$D$ in ${\mathcal M}_6$ mentioned above is characterized by the fact that it has at most four 
$g_4^1$'s.  Now, $\Af_5$ acts on these $g_4^1$Õs, and hence leaves them all invariant. Thus it preserves a map $C\to\Pb^1$ of degree 4. This is impossible since there are no invariant subset of ramification points. This proves the claim.

It is also known \cite{SB} that $S_C$ is unique. This uniqueness property implies that the faithful $\Af_5$-action on $C$, which extends  naturally to $\check\Pb (H^0(C, \omega_C))$,  will leave $S_C$ invariant.  A choice of an $\Af_5$-equivariant isomorphism $h:S_C\stackrel{\cong}{\to} S$ will then identify $C$ in an $\Af_5$-equivariant manner with a member of the Wiman-Edge pencil. Any two  $\Af_5$-equivariant isomorphisms $h, h': S_C\stackrel{\cong}{\to} S$ differ by an automorphism of $S$,  so by an element   $g\in \Sf_5$. But the $\Af_5$-equivariance then amounts to 
$g$ centralizing $\Af_5$. This can happen only when $g$ is the identity. So $h$ is unique.
\end{proof}

Let $\calB$ denote the base of the Wiman-Edge pencil (a copy of $\Pb^1$) so that we have projective flat morphism $\Cs\to\calB$. Recall that $\Sf_5$ acts on the family in such a manner that the action on  $\Cs\to\calB$ is through an involution $\iota$ which has two fixed points. We denote by  
$\calB^\circ\subset\calB$ the locus over which this morphism is smooth. 
So the restriction over $\calB^\circ$ is a family of smooth projective genus 6 curves endowed with a faithful $\Af_5$-action.  It has  the  following modular interpretation.

\begin{theorem}[{\bf Universal property}]
\label{thm:Wimanpencilmodular}
The  smooth part of the Wiman-Edge pencil, $\Cs_{\calB^\circ}\to \calB^\circ$,  is universal in the sense that every family $\Cc'\to \calB'$ of smooth projective genus 6 curves endowed with a fiberwise faithful $\Af_5$-action fits in a unique $\Af_5$-equivariant fiber square
\[
\begin{CD}
\Cc' @>>> \Cc_{\calB\circ}\\
@VVV @VVV\\
\calB' @>>> \calB^\circ
\end{CD}
\]
Moreover, the natural morphism $\calB^\circ\to \calM_6$ factors through an injection $\la\iota\ra\bs\calB^\circ\hookrightarrow \calM_6$. 
\end{theorem}
\begin{proof}
Theorem \ref{thm:onaquinticDelPezzo} (and its proof) has an obvious extension to families of genus $6$-curves with $\Af_5$-action. This yields the first assertion. If $t, t'\in\calB^\circ$ are such that $C_t$ and $C_{t'}$ are isomorphic as projective curves, then as we have seen, an isomorphism $C_t\cong C_{t'}$ is induced by an element of $\Sf_5$ and so $t'\in \{ t, \iota(t)\}$. 
\end{proof}

We will find in Subsection \ref{subsect:singmembers}  that the singular members of the Wiman-Edge pencil are all stable. 
We found already one such curve, namely the union of the 10 lines, and so this element will map in $\calM_6$ to the boundary.

From on we identify the base of Wiman-Edge pencil with $\calB$. 

\begin{corollary}\label{cor:Wimanpencilmodular}
The  Wiman curve $C_0$ is smooth and is $\Sf_5$-isomorphic to the curve  found in
Proposition \ref{prop:orbifolds}. It defines  the unique $\iota$-fixed point of $\calB^\circ$.
\end{corollary}
\begin{proof}
By  Theorem \ref{thm:Wimanpencilmodular}, there is a unique member of the Wiman-Edge pencil whose base point  maps the unique point of 
$\calB^\circ$ which supports a smooth genus 6 curve with $\Sf_5$-action. The Wiman-Edge pencil has two members with $\Sf_5$-action, one is
the union of the $10$ lines and  the other is the Wiman curve. So it must be the Wiman curve.
\end{proof}

\begin{corollary}\label{cor:action1}
Let  $C$ be a smooth projective curve genus $6$ endowed with a faithful action of $\Af_5$. If the resulting map $\phi:\Af_5\hookrightarrow\Aut(C)$ is not surjective, then it extends to an isomorphism $\Sf_5\cong \Aut(C)$, and $C=C_0$ is the Wiman curve. The $\Af_5$-representation (resp.\ $\Sf_5$-representation) $H^0(C, \omega_C)$ is equivalent to $I\oplus I'$ (resp.\ $E$) and  $H^1(C; \C)$  is equivalent to $I^{\oplus 2}\oplus I'{}^{\oplus 2}$ (resp.\ $E^{\oplus 2}$).
\end{corollary}
\begin{proof}
Since the $\Af_5$-curve $C$ is represented by a member of the Wiman-Edge pencil, we can assume it is a member of that pencil. Since $C$ is canonically embedded, an automorphism of $C$ extends naturally of the ambient projective space and hence also to $S$ (as it is the unique quintic del Pezzo surface containing $C$). This implies that $\Aut(C)\subset \Sf_5$. This inclusion is an equality precisely when $C$ is the Wiman curve.

As for the last assertion, we know that the representation space $H^0(C, \omega_C)$ is as asserted, since it is a member of the Wiman curve. The representations $I\oplus I'$ resp.\ $E$ are self-dual and 
since $H^1(C, \C)$ contains $H^0(C, \omega_C)$ as an invariant  subspace with quotient the (Serre-)dual of
$H^0(C, \omega_C)$, the last assertion  follows also. 
\end{proof}

\begin{remark}\label{rem:}
If an $\Af_5$-curve $C$ represents a point  $z\in\calB^\circ$, then $\Af_5$ acts nontrivially (and hence faithfully) on $H^0(C,\omega_C^{2})$ and hence
also on its Serre dual $H^1(C,\omega_C^{-1})$. Since $\Af_5$ is not   a complex reflection  group, a theorem of Chevalley implies that the orbit space 
$\Af_5\bs H^1(C,\omega_C^{-1})$ must be singular at  the image of the origin. The local deformation theory of curves tells us that  when $\Af_5$ is the full automorphism group of $C$,  the germ of  this orbit space at the origin is isomorphic to the germ of $\calM_6$ at  the image of $z$.  
Hence $\calB^\circ$ maps to the singular locus of $\calM_6$. (This is a special case of a theorem of Rauch-Popp-Oort which states that any curve of genus $g\ge 4$ with nontrivial group of automorphisms defines a singular point of $\calM_g$.) 
\end{remark}

\subsection{Singular members of the Wiman-Edge pencil}\label{subsect:singmembers}
The singular members of the Wiman-Edge pencil were found by W. Edge \cite{Edge2}. A modern proof of his result can be found in \cite[Theorem 6.2.9]{Cheltsov}.  Here we  obtain them in a different manner as part of a slightly stronger  result that we obtain with a minimum of computation.

Let us begin with an \textit{a priori} characterization of the reducible genus $5$ curves with $\Af_5$-action that occur in the Wiman-Edge pencil.

\begin{lemma}\label{lemma:specialconic}
There are precisely two reducible members $C_c$ and $C'_c$ of $\Cs$ distinct from $C_\infty$; each is is a stable union of  $5$ special conics whose intersection graph is $\Af_5$-equivariantly isomorphic to the full graph on  a $5$-element set.  Any element of $\Sf_5\ssm \Af_5$ exchanges $C_c$ and $C'_c$. 
\end{lemma}

\begin{proof} Let $C$ be such a member of $\Cs$ and let $Y$ 
an irreducible component of $C$. Then  $Y$ cannot be a line and so its degree $d:=-Y\cdot K_S$ in the anticanonial model must be $\ge 2$. Since $\Af_5$ has 
no subgroup of index $< 5$,  the number $r$ of irreducible components in the $\Af_5$-orbit of $Y$, must be $\ge 5$. But we also must 
have $rd\le \deg (C)=10$, and so the only possibility is that $(d, r)=(2,5)$.  

An irreducible, degree 2 curve in a projective space is necessarily 
a smooth conic. Its $\Af_5$-stabilizer has index $5$ in $\Af_5$, and so must be $\Sf_5$-conjugate to $\Af_4$. This implies that
$Y$ is a special conic. The 5 irreducible components of its $\Af_5$-orbit lie in distinct  conic bundles $\Pc_i$, and we number them accordingly 
$Y_1, \dots , Y_5$.  For $1\le i<j\le 5$,  any smooth member of $\Pc_i$  meets any smooth member of  $\Pc_j$ with multiplicity one, 
and so  this is in particular true  for $Y_i$ and $Y_j$.  Hence the set of singular points of $C$ is covered in a $\Af_5$-equivariant manner by the set 
of $2$-element subsets of $\{1, \dots, 5\}$.  The action of $\Af_5$ on this last set is  transitive, and so either the singular set of $C$ is a singleton 
$\{p\}$ (a point common to all the irreducible components) or the intersections $Y_i\cap Y_j$, $1\le i<j\le 5$ are pairwise distinct. 
The first case is easily excluded, for  $p$ must then be a fixed point of the $\Af_5$-action and there is no such point.  
So $C$ is as described by the lemma.

The sum of the classes of the 5 conic bundles is $\Sf_5$-invariant and hence proportional to $-K_S$. The intersection product with $-K_S$ is $5.2=10$ and hence this class is equal to $-2K_S$, in other words, the class of the Wiman-Edge pencil. Since an $\Af_5$-orbit of a special conic takes precisely one  member from every conic bundle, its follows that the sum of such an orbit indeed gives a member of  $\Cs$.
There two such orbits and so we get two such members.
\end{proof}

There are precisely two faithful projective representations of $\Af_5$ on $\Pb^1$ up to equivalence, and they only differ by precomposition with an automorphism of $\Af_5$ that is not inner. We will refer to theses two representations as the \emph{Schwarzian representations} of $\Af_5$.
Both appear as the symmetry groups of a regular icosahedron drawn on the Riemann sphere. This action has three irregular orbits, corresponding to the vertices, the barycenters of the faces and the midpoints of the edges, and so are resp.\ $12$, $20$ and $30$ in size. Their points come in (antipodal) pairs and the $\Af_5$-action preserves these pairs. We give a fuller discussion in Subsection \ref{subsect:icosahedralplane}.

\begin{lemma}\label{lemma:glueingconstruct}
There exists an irreducible stable curve $C$ of genus 6 with 6 nodes endowed with a faithful $\Af_5$-action. Such a  $C$ is unique up to an automorphism of $\Af_5$.
\end{lemma}
\begin{proof}
Let $C$ be such a curve with 6 nodes and denote by $D\subset C$ its singular set.
Then the normalization of $\hat C\to C$ is of genus zero: $\hat C\cong\Pb^1$. If $\hat D\subset \hat C$ denotes the preimage of $D$, then $\hat D$ consists of 12 points that come in $6$ pairs. The $\Af_5$-action on $\hat C$  lifts to $\hat C$ and will preserve  $\hat D$ and its decomposition into 6 pairs. From the above remarks it follows that $\Af_5$ acts on
 $\hat C$ as the symmetry group of an icosahedron drawn on  $\hat C$ which  has $\hat D$ as vertex set and such 
that  antipodal pairs are the fibers of $\hat C\to C$.  
This shows both existence and uniqueness up to an automorphism of $\Af_5$.
\end{proof}

\begin{proposition}
\label{prop:irred}
An irreducible singular member of the Wiman-Edge pencil is necessarily as in Lemma  \ref{lemma:glueingconstruct}: a stable curve with 6 nodes and  of geometric genus  zero. It appears in the Wiman-Edge pencil together with its outer transform.
\end{proposition}

\begin{proof} 
We have  already encountered the singular members $C_\infty$ and $C_c,C_c'$.  A well-known formula (see, for example, \cite{GriffithsHarris}, p.\ 509-510) applied to the Wiman  pencil gives
\[
e(S)-e(C)e(\Pb^1)+ C\cdot C =\sum_{t\in \calB}(e(C_t)-e(C))
\]
where $e(\; )$ stands for the Euler-Poincar\'e characteristic, $C$ denotes a general fiber and $C_t$ denotes a fiber over a point $t\in \calB$. The left hand side equals $27-(-20)=47$.
The reducible fibers $C_\infty$, $C_c$, $C'_c$ contribute  to the right hand side $15+2\cdot 10 = 35$, so that there is $12$ left as the contribution
coming the irreducible fibers. 

It is known that $e(C_t)-e(C)$ is equal to the sum of the Milnor numbers of singular points of $C_t$. 
We know that $\Af_5$ leaves invariant each fiber,  but that no fiber  other than $C_\infty$ (which is reducible) or the Wiman curve $C_0$ 
(which is smooth) is $\Sf_5$-invariant.  In other words, the irreducible fibers come in pairs.  Since $\Af_5$ cannot fix a point on $S$ (because it has no a nontrivial linear representations of dimension 2 and hence cannot act nontrivially on the tangent space at this point), and a proper subgroup of $\Af_5$ has 
index  $\ge 5$, the irreducible fibers  come as a single pair, with each member having  exactly $6$ singular points, all of Milnor number $1$, that is, having  
$6$ ordinary double points. Hence the normalization of such a fiber is a rational curve as in Lemma  \ref{lemma:glueingconstruct}
\end{proof}

We sum up with the following.

\begin{corollary} [{\bf Classification of singular members of $\Cs$}]
Each singular member of the Wiman-Edge pencil $\Cs$ is a stable curve of genus 6 with $\Af_5$-action. The set of these curves is a union of three sets:
\begin{description}
\item[(lines)] ${C_\infty}$, a union of $10$ lines with intersection graph the Petersen graph.
\smallskip

\item[(conics)] a pair $C_c$, $C'_c$, each of which is a union of  5 conics  whose intersection graph is the complete graph on $5$ vertices. 
\smallskip

\item[(irred)] a pair $C_{\ir}$, $C_{\ir}'$ of irreducible rational curves, each with 6 nodes.
\end{description}
\medskip

The action of $\Sf_5$ on $\Cs$ leaves ${C_\infty}$ invariant; the induced action on the set of 
$10$ lines of ${C_\infty}$ induces an action on the corresponding intersection graph that is isomorphic to the $\Sf_5$ action on the Petersen graph.  The action of any odd permutation of $\Sf_5$ on $\Cs$ interchanges $C_c$ and $C'_c$, and also interchanges $C_{\ir}$ and $C_{\ir}'$ .
\end{corollary}

\begin{remark}\label{rem:specialconicmoduli}
Our discussion in Subsection \ref{subsect:generalities} shows that $C_c\cup C'_c$, when regarded as a curve on $\calMc_{0,5}$, meets $\calM_{0,5}$ in the locus parameterizing $5$-pointed rational curves $(C; x_1, \dots, x_5)$ with the property that there exists an affine coordinate $z$ for $C$ such that
$\{ x, \dots, x_5\}$ contains the union of $\{0\}$ and the roots of $z^3=1$.
So we can characterize the Wiman-Edge pencil on $\calMc_{0,5}$ as the pencil which contains in $|\partial\calMc_{0,5}|$ and  these two loci.  
It is desirable  to have  a modular interpretation of this pencil.
\end{remark}

\subsection{Connection with the Dwork pencil}
Since the singular members of the Wiman-Edge pencil also play a special role in the work of Candelas \emph{et al.}\ \cite{Candelas},
this is perhaps a good place to make the connection with that paper.
Let $p=(p_1, \dots,p_5)$ be distinct points of $\Pb^1$. As Zagier \cite{Zagier} points out, there exist 
linear forms  $\ell_i$ on $\Cb^2$ defining  $p_i$, such that  
\[\ell_1{}^5+\dots+\ell_5{}^5=5\psi\ell_1\ell_2\ell_3\ell_4\ell_5\]
 for some $\psi\in \C$ with the $5$-tuple $(\ell_1{}^5, \ell_2{}^5,\ell_3{}^5,\ell_4{}^5,\ell_5{}^5)$ being unique up to a common scalar (so that $\psi^5$ only depends on $p_1, \dots,p_5$). In other words, 
 \[\ell:=[\ell_1:\cdots :\ell_5]: \Pb^1\to \Pb^4\] maps to a line  on a member of the \emph{Dwork pencil}, that is,  the pencil of quintic $3$-folds  $X_\psi\subset \Pb^4$ defined by  $z_1{}^5+\dots+z_5{}^5=5\psi z_1z_2z_3z_4z_5$, such that the coordinate hyperplane $z_i=0$ defines $p_i$.

This construction is essentially $\PGL_2(\C)$-equivariant and has some symmetries, perhaps the most obvious ones being its $\Sf_5$-symmetry.  But we also have acting the group $T[5]\cong \mu_5{}^4$ of order $5$  elements in the diagonal torus $T$ of  $\PGL_5(\C)$. Together with $\Sf_5$ this gives a faithful action of the subgroup  $\Sf_5\ltimes T[5]\subset \PGL_5(\C)$ on the total space of the pencil. It acts on the parameter $\psi$ through  a character $\Sf_5\ltimes T[5]\to \mu_5$, the latter being given as a nontrivial  $\Sf_5$-invariant character $\chi: T[5]\to \mu_5$. So $\ker (\chi)$ is as a  group isomorphic to $\mu_5{}^3$ and  every $X_\psi$ is stabilized by $\Sf_5\ltimes \ker (\chi)$. A good way to express this is that  we have thus obtained the following: 
\begin{enumerate}
\item [(i)] an unramified   $T[5]$-covering $\Tc_{0,5}\to \calM_{0,5}$ endowed with an action $\Sf_5\ltimes T[5]$ which extends the $T[5]$-action and is compatible with the  $\Sf_5$-action on $\calM_{0,5}$,
\item [(ii)] a regular function $\psi: \Tc_{0,5}\to \C$, equivariant with respect to the above character (so that $\psi^5$ is defined as an $\Sf_5$-invariant function on $\calM_{0,5}$),
\item [(iii)] an $\Sf_5\ltimes T[5]$-equivariant lift of $\psi$ from $\Tc_{0,5}$ to the Fano variety  $\text{Fano}(\Xs/\Pb^1)$ of lines on the Dwork pencil.
\end{enumerate}
The main results of \cite{Candelas}  can then be summed up as follows:

A fiber of $\calM_{0,5}\xrightarrow{\psi^5}\C$ is the sum of a member of the Wiman pencil plus its transform under an odd permutation and for a suitable  parametrization of  the pencil base $\calB$ by a parameter $\phi$ such that $C_0$ resp.\  $C_\infty$ is defined by $\phi=0$ resp.\ $\phi=\infty$ and we have  $32\psi^{-5}=\phi^2+3/4$. Moreover,
the    morphisms above with their symmetry extend (uniquely) over the blowup $\Cs\to \calMc_{0,5}$ of $\calMc_{0,5}$ in $\Delta$ (recall that this is simply the total space of the Wiman-Edge pencil $\phi:\Cs\to\Pb^1$):  we have a $T[5]$-cover $\pi:\tilde\Cs\to \Cs$ with a compatible $\Sf_5\ltimes T[5]$-action  and an equivariant extension of the Fano map  to $\tilde\Cs$ which makes  the Dwork pencil and the symmetrized Wiman-Edge pencil fit in  the following  commutative diagram
\[
\xymatrix{
 \text{Fano}(\Xs/\Pb^1) \ar[rd]&\tilde\Cs\ar[r]|-{\pi}\ar[l]|-{F} \ar[d]|-{\psi} & \Cs \ar[d]|-{\phi^2}\\
&\Pb^1 \ar[r]  &\Pb^1\\}
\]
Here  the slant map is the  structural map  and the bottom map is given by $\psi\mapsto 32\psi^{-5}-3/4$.
For $\psi\not=0, \infty$, $\tilde C_\psi$ parametrizes all the nonisolated  points of $\text{Fano}(X_\psi)$; for $\psi^5\not=2^7/3$ it consists of two irreducible curves, each of which  is an unramified $\ker (\chi)$-cover of $C_{\pm\phi}$ (of arithmetic genus $626$), but  for $\psi^5=2^7/3$, there is only one irreducible curve (which covers the Wiman curve).

The singular members of the Wiman-Edge pencil are accounted for as follows:
the Fermat quintic ($\psi=0$)  yields the sum of 10 lines with multiplicity $2$ ($2C_\infty$), the sum of the $5$ coordinate hyperplanes ($\psi=\infty$) yields  $C_c+C'_c$, and if $\psi$ is a $5$th root  of unity, we get  $C_\ir+C'_\ir$.

\subsection{Plane model of the Wiman-Edge pencil}\label{planemodel}
Let $\pi:S\to \Pb^2$ be the blowing-down morphism. The Wiman-Edge pencil is the proper inverse transform of a pencil of curves of degree $6$ with points of multiplicity $\ge 2$ at the fundamental points invariant with respect to the Cremona group $G$ of transformations isomorphic to $\Af_4$. Following Edge, we chose  the fundamental points of $\pi^{-1}$ to be the  points 
\[(-1:1:1), \ (1:-1:1), \ (1:1:-1), \ (1:1:1).\] 
The group generated by  projective transformations that permute the coordinates and their sign changes is isomorphic to $\Sf_4$, where the sign changes are even permutations of order $2$.  Together with the  additional symmetry  defined by the Cremona involution $T$ with the first reference points, we obtain a subgroup of Cremona transformations isomorphic to $\Sf_5$. The subgroup $\Af_5$ is generated by cyclic permutations, sign changes and the transformation $T$. 

Let $F = 0$ be the equation of a curve from the Wiman-Edge pencil. The condition that $F$ is invariant with respect the sign changes shows that  $F$  must be a combination of monomials $x^ay^bz^c$ with $a,b,c$ even. This allows us to write 
$$F= a(x^6+y^6+z^6)+b(x^4y^2+y^4z^2+z^4x^2)+c(x^4z^2+y^4x^2+z^4y^2)+dx^2y^2z^2.$$
The additional conditions that $3a+3b+3c+d= 0$  will guarantee that the fundamental points are singular points. The Cremona involution $T$ is given by
{\Small $$\sigma:(x:y:z)\mapsto (-x^2+y^2+z^2+xy+xz+yz:x^2-y^2+z^2+xy+xz+yz:x^2+y^2-z^2+xy+xz+yz)$$}
The invariance of $F$ with respect to $T$ gives  $(a,b,c)$ is a linear combination of $(2,1,1)$ and $(0,1,-1,0)$. This gives the equation of the Wiman-Edge pencil
{\Small $$
 F = \lambda (x^6+y^6+z^6+(x^2+y^2+z^2)(x^4+y^4+z^4)-12x^2y^2z^2)+\mu (x^2-y^2)(y^2-z^2)(x^2-z^2) = 0
$$}
We check that the Wiman curve $B$ is the member of the pencil with $(\lambda:\mu) = (1:0)$. Computing the partial derivatives, we find that the curve is indeed smooth confirming  Corollary \ref{cor:}.

\begin{figure}[h]
\begin{subfigure}{.4\textwidth}
\includegraphics[scale=0.35]{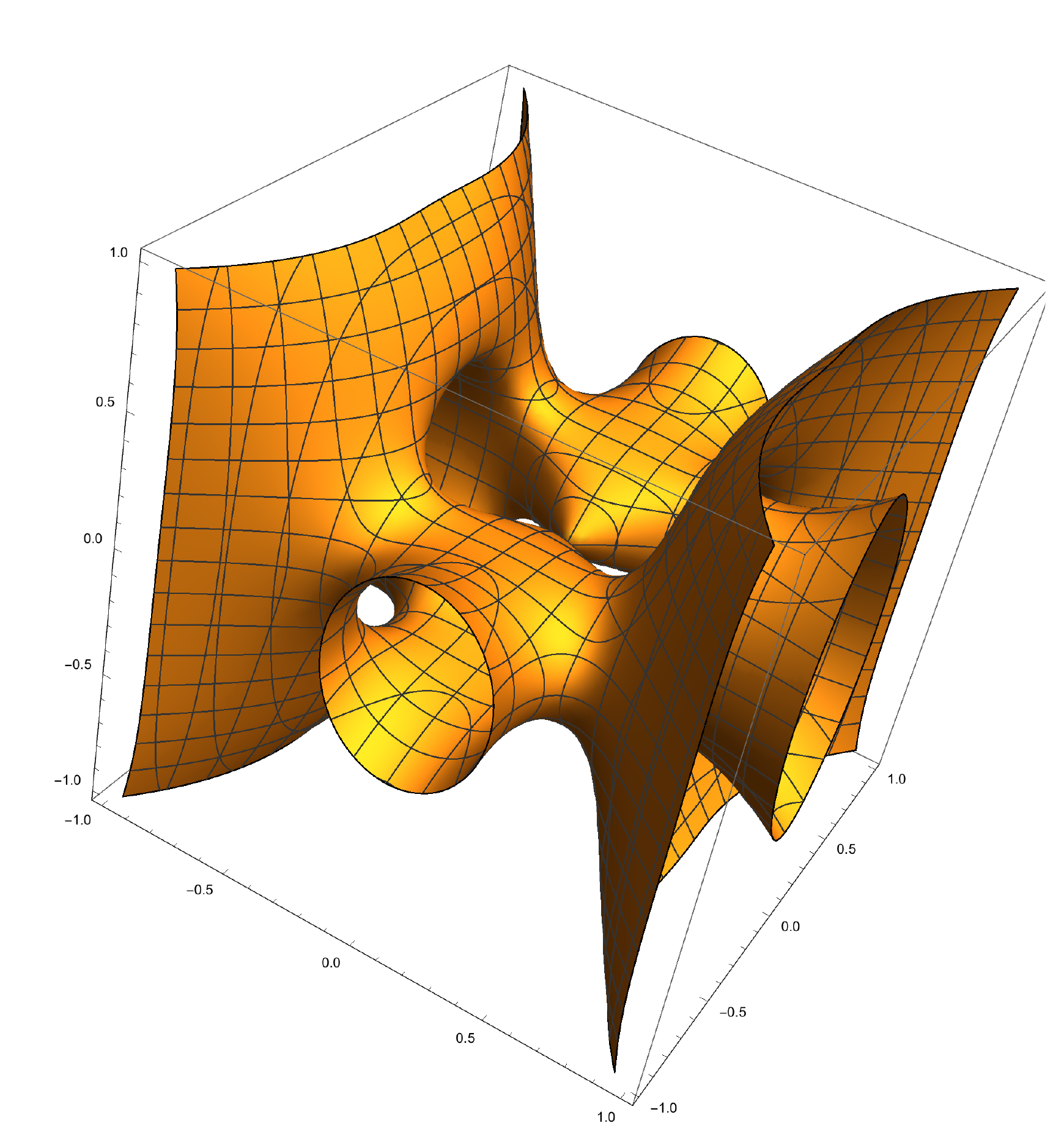}
\end{subfigure}
\begin{subfigure}{.4\textwidth}
\includegraphics[scale=0.35]{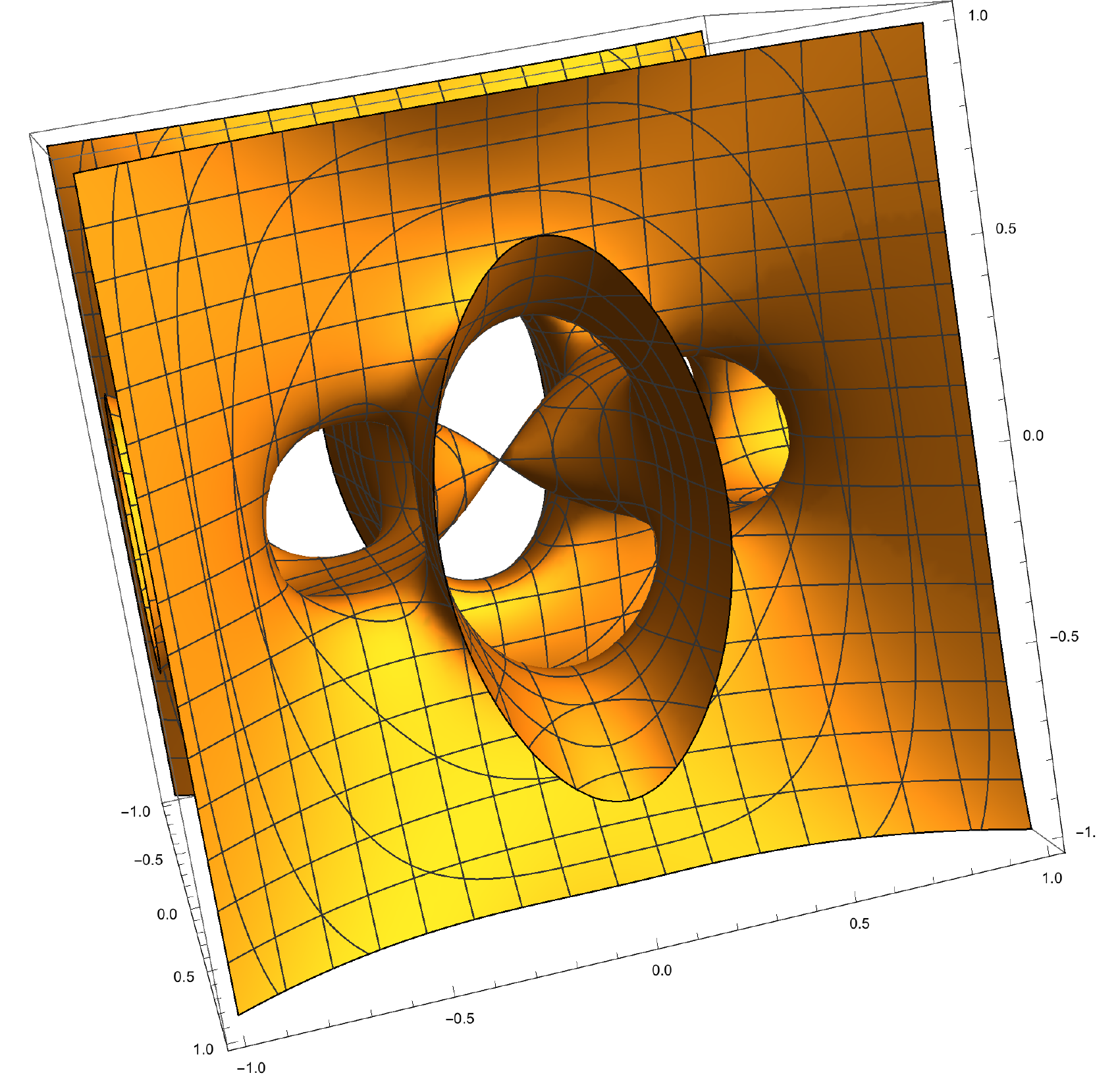}
\end{subfigure}
\caption{\small Two views of a piece of the Wiman sextic (the image of the Wiman curve in $\Pb^2$).}
\label{figure:wiman1}
\end{figure}

The base locus of the pencil should be  a subscheme of length $6^2 = 36$. The four reference points contribute $16$ to this number, the rest is the image of 20 base points of the Wiman-Edge pencil on $S$ in the plane. It easy to find them. Each line through two fundamental points intersects $C_0$ at two nonsingular points, this gives us $12$ points:
\[(\sqrt{-3}:\pm 1:\pm 1)\ \text{and}\ (\pm 1:\sqrt{-3}:\pm 1)\ \text{and}\   (\pm 1:\pm 1:\sqrt{-3}).\]
The remaining 8 points are on the exceptional curves where they represent the directions of the tangent lines to the branches of singular points of $B$. The equations of the 4 pairs of these tangent lines are
\begin{eqnarray}\label{conics}
(x+\epsilon y+\epsilon^2 z)(x+\epsilon^2 y+\epsilon z) &=& 0,\\ \notag
(x+\epsilon y-\epsilon^2 z)(x+ \epsilon^2 y-\epsilon z) &=& 0,\\ \notag
(x-\epsilon y+\epsilon^2 z)(x- \epsilon^2 y+\epsilon z) &=& 0,\\ \notag
(x-\epsilon y-\epsilon^2 z)(x-\epsilon^2 y-\epsilon z) &=& 0, \notag
\end{eqnarray}
where $\epsilon = \frac{-1+\sqrt{-3}}{2}$ is a primitive root of $1$. One checks that  each tangent line passing through $p_i$ together with three of the lines $\la p_i,p_j\ra, j\ne i,$ form a harmonic 4-tuple of lines in the pencil of lines through $p_i$ (\footnote{Edge calls them \emph{Hesse duads} for the following reason. It follows from 
observing the Petersen graph the stabilizer subgroup of each line acts on the line by permuting three intersection points with other lines. If one considers these points as the zero set of a binary form $\phi$ of degree 3, then 
its Hessian binary form of degree 2 given as the determinant of the matrix of second partial derivatives of $\phi$ has zeros at two point, the Hessian duad.}).

Choose one of the branch tangents, say $x+\epsilon y+\epsilon^2 z = 0$. It intersects the Wiman sextic $B$ at two nonsingular points $(1:\sqrt{-3}:1)$ and $(1:-1:\sqrt{-3})$. We see that each double point of $W$ is a \emph{biflecnode}, i.e. the branch tangents at singular point intersect the curve at this point with multiplicity $4$ instead of expected $3$.

The singular irreducible members $C_{\ir}$ (resp.  $C_{\ir}'$) are proper transforms of the members of the pencil corresponding to the parameters $(\lambda:\mu) = (1:5\sqrt{5})$ (resp. $(1:-5\sqrt{5})$). The singular points of $C_{\ir}$ besides the fundamental points  are
$$(0:0:\lambda), \ (0:1:-\lambda), \ (1:0:\lambda'),\ (1:0:-\lambda'),\ (1:\lambda:0),\ (1,-\lambda,0),$$
where $\lambda = \frac{1+\sqrt{5}}{2}, \lambda' = \frac{1-\sqrt{5}}{2}$. The six singular points of $C_{\ir}'$ are obtained from these points by replacing $\lambda$ with $\lambda'$.

The reducible members  $C_c$ and $C_c'$ of the Wiman-Edge pencil are proper transforms of the members of the  pencil of plane curves
$$(x+\epsilon y+\epsilon^2 z)(x+ \epsilon y-\epsilon^2 z)(x-\epsilon y+\epsilon^2 z)(x-\epsilon y-\epsilon^2z)(x^2+\epsilon^2 y^2+\epsilon z^2) = 0,$$
where $\epsilon\ne 1, \epsilon ^3 = 1$.  As shown already by Edge, they correspond to parameters $(\lambda:\mu) = (1:\pm \sqrt{-3})$.
We leave these computations to the reader, they are straightforward.

\subsection{Irregular orbits in $S$}\label{subsection:fixedpoints} 
Recall that if a group acts on a set, an orbit is called \emph{regular} if the stabilizer of one (and hence any) of its points is trivial; otherwise it is called an \emph{irregular} orbit.
For what follows it will helpful to have a catalogue of irregular  an $\Sf_5$-orbits and $\Af_5$-orbits in $S$.  Here we can observe that a $\Sf_5$-orbit  is an $\Af_5$-orbit if and only if its $\Sf_5$-stabilizer is not contained in $\Af_5$ (otherwise it splits into two $\Af_5$-orbits). So a determination of the 
irregular $\Sf_5$-orbits determines one of the irregular $\Af_5$-orbits. The $\Sf_5$-equivariant incarnation of $S$ as the moduli space of  stable, $5$-pointed genus zero curves makes this  determination (which is in fact due to Coble) rather straightforward, as we will now explain.  

A point of $\calMc_{0,5}$ is the same thing as  a \emph{stable} map  
$x:\{1,2 \dots, 5\}\to \Pb^1$ (where `stable' means here that every fiber has at most two elements),   given up to a composition with a M\"obius transformation. The $\Sf_5$-stabilizer of a stable map $x$ consists of the set of $\sigma\in \Sf_5$ for which  there exists a  $\rho(\sigma)\in \PGL_2(\C)$ with the property that $x\sigma=\rho(\sigma)x$. Since $x$ is stable, its image has at least $3$ distinct points,  and  so $\rho(\sigma)$ will be unique. It follows that $\rho$ will be a group homomorphism. Its image will be a finite subgroup of $\PGL_2(\C)$
with the property that it has in $\Pb^1$ an orbit of size $\le 5$.

Klein determined the finite subgroups of $\PGL_2(\C)$ up to conjugacy:  they are the cyclic groups,  represented by  the group
$\mu_n$ of $n$th roots of unity  acting in $\C\subset \Pb^1$ as scalar multiplication; the dihedral groups, represented  by the semidirect product of 
$\mu_n$ and  the order 2 group  generated by the inversion $z\mapsto z^{-1}$; and the 
tetrahedral, octahedral and icosahedral  groups, which are isomorphic to $\Af_4$, $\Sf_4$, $\Af_5$ respectively.  The  octahedral and icosahedral groups have no orbit of size $\le 5$ in $\Pb^1$,  and hence cannot occur here. The tetrahedral group  has one such orbit: it is of size $4$ (the vertices of a tetrahedron), but since we want  a degree $5$ divisor, it then  must have a fixed point, and this is clearly not the case. 

It remains to go through the dihedral  and cyclic cases. We  denote by $\Cf_{n}$ a cyclic group of order $n$ and by $\Df_{2n}$ the dihedral group of order $2n$ isomorphic $\Cf_n\rtimes \Cf_2$ (\footnote{We follow the now standard ATLAS notation for finite groups.}). The  conjugacy classes of the nontrival cyclic subgroups of $\Af_5$ are represented by  $\la (12)(34)\ra\cong \Cf_2$ and 
$\la (123)\ra\cong \Cf_3$, and  for  $\Sf_5$ we have the two additional classes represented by $\la (12)\ra\cong \Cf_2$ and $\la (12)(345)\ra\cong \Cf_6$. Likewise, there are 3 conjugacy  classes of dihedral subgroups  in $\Af_5$ represented by 
\[
\la (12)(34),(13)(23)\ra \cong \Df_4,\  \la (12)(45), (123)\ra\cong \Df_6 \cong \Sf_3,\  \la (12)(34),(12345)\ra\cong \Df_{10} 
\]
and an additional three conjugacy classes in $\Sf_5$: 
\[
\la (12),(34)\ra \cong \Df_4,\ \  \la (12),(123)\ra \cong \Df_6,\ \la (12),(1324)\ra\cong \Df_8.
\]
To distinguish the conjugacy classes of dihedral subgroups of order $4$ or $6$ we denote them by 
$\Df_{4}^{\ev}, \Sf_3^{\ev}$ (resp. $\Df_{4}^{\odd}, \Sf_3^{\odd}$) if they are contained in $\Af_5$ (resp.\ in $\Sf_5$, but not in $\Af_5$).  A similar convention applies to $\Cf_2$: we have  $\Cf_2^{\odd}$ and $\Cf_2^{\ev}$.

We then end up with the following list, due to  A. Coble \cite{Coble}, pp. 400--401.

\begin{theorem}[{\bf The irregular orbits of $\Sf_5$ acting on $S$}]
\label{thm:irrorbits} The set of irregular orbits of $\Sf_5$ acting on $S$ is one of the following sets, 
named by the conjugacy class of a stabilizer subgroup :
\begin{description}
\item[$\Cf_2^{\odd}$] For example, $\la(12)\!\ra$ is the stabilizer of $(0,0,\infty, 1,z)$ when $z$ is generic. An orbit of this type has size $60$. 
It is a regular $\Af_5$-orbit.
\smallskip

\item[$\Cf_2^{\ev}$] For example, $\la (12)(34)\!\ra $ is the stabilizer of $( z, -z, 1, -1, \infty)$ when $z$ is generic.  An orbit of this type has size $60$ and decomposes into two $\Af_5$-orbits of size $30$.
\smallskip

\item[$\Cf_4$] For example, $\la (1234)\!\ra$ is the stabilizer of  $(1,\sqrt{-1}, -1,-\sqrt{-1},\infty)$. This is an $\Sf_5$-orbit of size $30$. It is also an $\Af_5$-orbit of type $\Cf_2^{\ev}$ (take $z=\sqrt{-1}$).
\smallskip

\item[$\Df_4^{\odd}$] For example  $\la (12),(34)\!\ra $ is the stabilizer of  $(0,0, 1 , -1,\infty)$. This is a $\Sf_5$-orbit of size $30$, which is also an $\Af_5$-orbit of type $\Cf_2^{\ev}$ (let $z\to 0$).
\smallskip

\item[$\Sf_3^{\ev}$] For example  $\la (23)(45),(123)\!\ra $ is the stabilizer of  $(1, \zeta_3, \zeta_3^2, 0,\infty)$. This is a $\Sf_5$-orbit 
of size $20$ which splits into two $\Af_5$-orbits of size $10$.
\smallskip

\item[$\Df_8$]  For example,   $\la (12),(1324)\!\ra $ is the stabilizer of  $(0, 0,\infty,\infty, 1)$. This is the unique $\Sf_5$-orbit  of size $15$. It is also an 
$\Af_5$-orbit of type $\Df_4^{\ev}$.
\smallskip

\item[$\Cf_6$]  For example, $\la (12)(345)\!\ra $ is the stabilizer of  $(\infty, \infty, 1, \zeta_3, \zeta_3^2)$. This  is a single orbit of size $20$ which is also a $\Af_5$-orbit of type $\Cf_3$.
\smallskip

\item[$\Df_{10}$] For example, $\la (25)(34),(12345)\!\ra$  is the stabilizer of  $(1, \zeta_5, \zeta_5^2, \zeta_5^3, \zeta_5^4)$.
The associated orbit is  of size $12$ and splits into two $\Af_5$-orbits of size $6$.
\end{description}

\medskip
In particular, the irregular $\Af_5$-orbits are of type  $\Cf_2^{\ev}$ are parametrized by a punctured rational curve.
\end{theorem}

\begin{remark}\label{rem:irrcompA2locus}
It is clear from Theorem \ref{thm:irrorbits} that we have a curve of irregular $\Af_5$-orbits of size $30$. However the locus of such points in $S$ has  $15$ irreducible components. This is because the preimage of a $\Af_5$-orbit under the map  $w\in\C\mapsto (w,-w,1,-1, \infty)$ is generically of the form 
$\{z, 1/z\}$.  An example of the closure such an irreducible component is the preimage  of the line defined in $\Pb^2$ by $t_2=t_0+t_1$ under the blowup of the vertices of the coordinate vertex (it  is pointwise fixed  under an even linear permutation of these vertices). Since this line does not pass through any 
of the four vertices, this also shows that its preimage in $S$  is a rational normal curve of degree $3$.
We thus obtain a $\Sf_5$-invariant curve on $S$ of degree $45$ (defined by a section of $\omega_S^{-9}$) with $15$ irreducible components. \textit{A  priori} this section is 
$\Af_5$-invariant, but as Clebsch \cite{Clebsch} showed,  it is in fact $\Sf_5$-invariant (see also Remark \ref{rem:clebsch}).
\end{remark}

Since we have already determined the size of some orbits in the  anticanonical model, we can now interpret the  orbits thus found. We do  this only 
insofar it concerns $\Af_5$-orbits, because that is all we need. Note that this is not a closed subset. Here is what remains, but stated in terms of the Wiman-Edge pencil:

\begin{corollary}\label{cor:irrorbits} The irregular $\Af_5$-orbits in $S\cong\calMc_{0,5}$ of size $30$ are parametrized by punctured rational curve and 
the others (named by the conjugacy class of a stabilizer subgroup) are as follows:
\begin{description}
\item[$\Cf_3$]  This $20$-element orbit is the base point locus $\Delta$   of $\Cs$. 
\item[$\Df_4^{\ev}$] This $15$-element orbit is the  singular locus of $C_\infty$. 
\item[$\Sf^{\ev}_3$]  This consists of two $10$-element orbits in $\calM_{0,5}\cong S\ssm C_\infty$ equal to  $\Sing(C_c)$ and $\Sing(C'_c)$. 
\item[$\Df_{10}$]  This consists of two $6$-element orbits, namely the singular loci of the two irreducible members of the Wiman-Edge pencil,
 $\Sing(C_{\ir})$ and  $\Sing(C_{\ir}')$.
\end{description}
The orbit pairs of type   $\Sf^{\ev}_3$  and $\Df_{10}$ are swapped by an a conjugacy with an element of $\Sf_5\ssm\Af_5$ (which induces a nontrivial 
outer automorphism of $\Af_5$).
\end{corollary}
\begin{proof}
Theorem \ref{thm:irrorbits} yields a complete list of the irregular  $\Af_5$-orbits in terms of $\calMc_{0,5}$ of size smaller than $30$. We have already encountered
some of these orbits as they appear in this corollary. All that is then left to do is to compare cardinalities.
\end{proof}

\section{Projection to a Klein plane}\label{sect:deg5cover} 
\subsection{The Klein plane}\label{subsect:icosahedralplane} 
The two representations $I$ and $I'$ of $\Af_5$ are  the complexification of two real representations that realize $\Af_5$ as the group of motions of a regular icosahedron.  They differ only 
in the way we have identified this group of motions with $\Af_5$.  The full group of isometries of the regular icosahedron (including reflections) is a direct product $\{\pm 1\}\times \Af_5$, and is in fact a Coxeter group of type $H_3$, a property that will be quite helpful to us when we need to deal with the $\Af_5$-invariants in the symmetric algebra of $I$.
Both $\Af_5$ actions give rise to $\Af_5$ actions on the unit sphere in Euclidean 3-space. Via the
isomorphism  $\SO_3(\Rb)\cong \PU_2$, they can also be considered as actions of $\Af_5$ on the Riemann sphere $\Pb^1$. These projective representations are what we have called 
the {\em Schwarzian representations} and are the only two nontrivial  projective representations of $\Af_5$ on $\Pb^1$ up to isomorphism. 

We observed above that a Schwarzian representation has 3 irregular orbits of sizes $12$, $20$ and $30$, corresponding respectively to the vertices, 
the barycenters and the midpoints of the edges of a spherical icosahedron. The antipodal map, when considered as an involution $\Pb^1$, is 
antiholomorphic: it comes from assigning  to a line in $\C^2$ its orthogonal complement. We can think of  this as defining a  $\Af_5$-invariant real structure on  $\Pb^1$ without real points. In particular, the involution is not in the image of $\Af_5$.   
Yet it preserves the $\Af_5$-orbits, so that each orbit decomposes into pairs.  The preimages of $\Af_5\hookrightarrow \Aut(\Pb^1)$  under the degree 2 isogeny $\SU_2\to \PU_2$ define two representations of degree $2$ of  an extension $\hat \Af_5$ of $\Af_5$ by a central subgroup of order $2$, called \emph{binary icosahedral group}. As above, these two representations of $\hat\Af_5$ differ by an outer automorphism.  If we take  the symmetric square of such a  representation, then
the kernel $\{ \pm 1\}$ of this isogeny  acts trivially and hence factors through a linear representation of  $\Af_5$ of degree 3. 
This is an icosahedral representation  of type $I$  or $I'$. 

We can phrase this solely in terms of a given Schwarzian representation of $\Af_5$  on a projective line $K$. For then the projective plane $P$ underlying the associated  icosahedral representation is the one of 
effective degree 2 divisors on $K$ (the symmetric square of $K$) and $K$ embeds $\Af_5$-equivariantly in $P$ as the locus defined by 
points with multiplicity 2. This is of course the image of $K$ under the Veronese embedding and makes $K$ appear as a conic.

We will identify $K$ with its image in $P$, and following Klein \cite{Klein} we refer to this image as the \emph{fundamental conic}. It is also defined by a non-degenerate 
$\Af_5$-invariant quadratic form on the icosahedral representation (which we know is self-dual). We call $P$ a \emph{Klein plane}.
So  an $x\in P\ssm K$ can be understood as a $2$-element subset of $K$. The latter spans a line in $P$ and this is simply the polar that
the conic $K$ associates to $x$ (when $x\in K$, this will be  the tangent line of $X$ at $x$).

Following Winger,  we can now identify all the  irregular $\Af_5$-orbits in $P$.
As we have seen, $K$ has exactly three irregular $\Af_5$-orbits having sizes $30$, $20$ and $12$, each of which being invariant  under an antipodal 
map. This  antipodal  invariance implies that the antipodal pairs in the above orbits span  a collection of resp.\ $15$, 
 $10$ and $6$ lines in $P$, each of which makes up an $\Af_5$-orbit. When we regard these pairs  as  effective divisors of degree 2, they also yield 
 $\Af_5$-orbits in $P\ssm K $ of the same size ($K$ parameterizes the nonreduced divisors). The bijection between lines and points is induced by 
 polarity with respect to $K$. This  yields all the irregular orbits in the Klein plane: 

\begin{lemma}[Winger \cite{Winger}, \S1]\label{invset1} 
There are unique irregular $\Af_5$-orbits in $P$ having size $12$, $20$ (both in $K$),  $6$, $10$ or $15$ (all three in $P\ssm K$). The remaining irregular orbits in $P$  have size $30$ and are parametrized by a punctured rational curve.  Further, the points with stabilizer  a fixed $\tau\in\Af_5$ of order $2$ is  open and dense in the image of the map $K\to P$ given by $z\in K\mapsto (z) +(\tau(z))$. 
\end{lemma}
\begin{proof} 
We can think of a point of $K$ as a point on the  icosahedron in Euclidean 3-space. An element of  $P\ssm K$  is represented by an effective  degree 2 divisor on $K$
which has the same $\Af_5$-stabilizer. If we identify $K$ with its $\Af_5$-action as a spherical icosahedron in Euclidean 3-space with its group of motions, then such a divisor spans an affine line in Euclidean 3-space  with the same stabilizer. When the line passes through the origin, we get the three orbits 
of sizes $30$, $20$ and $12$, otherwise the stabilizer is of order $2$. 
\end{proof}

We call a member of the $6$-element $\Af_5$-orbit in $P\ssm K$, a   \emph{fundamental point} and denote this orbit by $F$. We call the polar line of such a point a \emph{fundamental line}.

\subsection{Two projections}\label{subsect:kleinplane}
The irreducible $\Sf_5$-representation $E_S$ splits into two $3$-dimensional irreducible $\Af_5$-representations $I_S$ and $I'_S$. The two summands 
give a pair of disjoint planes $\check\Pb (I_S)$ and $\check\Pb (I'_S)$ in $\check\Pb (E)$, to which we shall refer as \emph{Klein planes}. We abbreviate them by $P$ resp.\ $P'$, and  denote by $K\subset P$ and $K'\subset P'$ the fundamental ($\Af_5$-invariant) conics. We have $\Af_5$-equivariant (Klein) projections  $p:S\dasharrow P$  and $p':S\dasharrow P'$ of the anticanonical model $S\subset \check\Pb(E)$ with center $P'$ resp.\  $P$. Precomposition with an  element of $\Sf_5\ssm\Af_5$  exchanges these projections.

\begin{proposition}\label{prop:Kleinproj} 
The Klein planes  are disjoint with $S$  and the Klein projections are  finite morphisms of degree 5. Together they define  a finite morphism
\[f:=(p, p'): S\to P\times P'\] that is birational onto its image.
\end{proposition}

The first assertion of Proposition~\ref{prop:Kleinproj} is Theorem 5 of \cite{Slodowy}.

\begin{proof} 
We focus on $p:S\to P$.  Let us first note that $P'\cap S$ is $\Af_5$-invariant and equals the base locus of the linear system of anticanonical curves parametrized by $\Pb (I)\subset\Pb (E)$. The curve part $Y$ of $P'\cap S$ has degree $<5$. The class $[Y]\in H^2(S; \Z)$ of $Y$ is
$\Af_5$-invariant. Since  the span of $K_S$ is supplemented  in $H^2(S; \C)$  by an irreducible  $\Af_5$-representation isomorphic to $V$, 
we  must have  $[Y]=-dK_S$ for some integer $d\ge 0$.  But then 
\[5>\deg(Y)=[Y]\cdot (-K_S)= 5d\] and so $Y=\emptyset$. It follows that the base locus is finite. Since it is
a linear section of $S$ it has at most 5 points. But Lemma \ref{invset1}  shows that  $P'$ has no orbit of size $\le 5$ and so the  base locus  is empty.

To see that $p:S\to P$ is surjective, suppose its image is a curve, say of degree $m$. The preimage of a  general line in $P$ in $S$ is an anticanonical curve and hence connected. This implies that $m=1$. But then $S$ lies in hyperplane and this is a contradiction.

So $p$ is  a surjection of nonsingular surfaces. If some irreducible curve is contracted by $p$, then this curve will have negative self-intersection. The only curves on $S$ with that property are the lines,  and then a line is being contracted. Since  $\Af_5$ acts  transitively on the lines, all of them are then  contracted.
In other words, the exceptional set is the union $C_\infty$ of the 10 lines on $S$. But  $C_\infty$ has self-intersection $(-2K_S)^2=20>0$ and hence
cannot be contracted.

So the preimage of a  point in $P$ is finite. This is also the intersection of the quintic surface $S$ with a codimension 2 linear subspace  and so this fiber consists of 5 points, when counted with multiplicity.

For the last assertion, we notice  that since one of the components of $f$ has degree $5$, the degree of $S\xrightarrow{f}f(S)$  must divide $5$. So it is either $1$ or $5$. If it is  $5$ then $p$ and $p'$ will have the same generic fiber
so that $p'$ factors through $p$ via an isomorphism $h:P\xrightarrow{\cong}P'$.  But this would make the $\Af_5$-representations 
$I$ and $I'$ projectively equivalent, which is not the case.  Alternatively, (We could alternatively observe that then the elements of $\Sf_5\ssm\Af_5$ preserve the fibers of $p$  so that we get in fact an action of $\Sf_5$ on $P$ which makes $p$ $\Sf_5$-equivariant. But there is no projective representation of $\Sf_5$ on $\Pb^2$.)
\end{proof}  

The proof of the next proposition makes use of the Thom-Boardman polynomial for the $A_2$-singularity locus. Let us first state the general result that we need. Let $f$ be a morphism between two compact, nonsingular complex surfaces.
Assume first that  $f$ has a  smooth locus of  critical points ($=$ ramification divisor) $\Sigma^1(f)$. As Whitney and Thom observed, $f|\Sigma^1(f)$  need not be a local  immersion: generically it will  have a finite set $\Sigma^{1,1}(f)$ of (Whitney) cusp singularities; at such a point $f$ exhibits a stable map-germ which in local-analytic coordinates can be given by $(x,y)\mapsto (x^3+xy, y)$. 

In  the more general case when $\Sigma^1(f)$ is a reduced divisor (when defined by the Jacobian determinant),  
$\Sigma^{1,1}(f)$ is defined as a $0$-cycle on the source manifold.
While is not so hard to prove that the degree of $\Sigma^{1,1}(f)$ is a characteristic number of the virtual normal bundle $\nu_f$ of $f$, it is another matter to obtain a closed formula for it. In the present case it is equal to  $\la c_1^2(\nu_f) + c_2(\nu_f), [S]\ra$ (see for instance
 \cite[Theorem 5.1]{rimanyi}, where this is listed as the case $A_2$).

\begin{proposition}[{\bf The ramification curve}]
\label{prop:Kleinram} The
ramification curve $R$ of the finite morphism $p:S\to P$ is a singular irreducible member of the Wiman-Edge pencil, and hence obtained 
as a $\Af_5$-curve by means of the procedure of  Lemma \ref{lemma:glueingconstruct}. 
\end{proposition}
\begin{proof}
The divisor class $[R]$ is  given by the well-known formula 
$K_S-p^*(K_{P})$. Since $K_{P}$ is $-3$ times the class of a line, and $p^*$ takes the class of a line to
the class of hyperplane section (i.e., $-K_S$), we have $p^*(K_{P})=3K_S$. It follows that $[R]=-2K_S$. Since  $R$ is $\Af_5$-invariant, it must be a member of the Wiman-Edge pencil. In particular, $R$ is reduced.

We now apply the above formula to $p: S\to P$. The  Chern classes of the virtual normal bundle of $p$
are easily computed: $c_1(\nu_p) = 2K_S$ and $c_2(\nu_p)$ takes the value  $-2$ on the fundamental class $[S]$ of $S$.
We find that 
\[\la c_1(\nu_p) ^2 + c_2(\nu_p), [S]\ra = 4\cdot 5-2 = 18.\] So
$\Sigma^{1,1}(p)$ is a $0$-cycle whose support  is a $\Af_5$-invariant subset of $S$ contained in $R$.
In Subsection \ref{subsection:fixedpoints}, we found that  $\Af_5$ has
two $6$-element orbits in $S$ (each being the singular locus of an irreducible member of the Wiman-Edge pencil) and that
any other $\Af_5$-orbit has  at least $10$ elements. Since $R$ is a member of the Wiman-Edge pencil, we conclude that
that it must be one with 6 nodes (and that each node has multiplicity $3$ in $\Sigma^{1,1}(p)$). This also proves that 
$p|R$ is a local isomorphism at the $20$-element orbit $\Delta$.
\end{proof}

\begin{remark}\label{computation2} 
The assertion of Proposition~\ref{prop:Kleinram} can be confirmed by computation. If $\pi:S\to \Pb^2$ is the blowing down morphism and $\phi:S\to\Pb^5 = \Pb(E)$ is the $\Sf_5$-equivariant anti-canonical embedding, then  the composition $\phi\circ \pi^{-1}:\Pb^2\dashrightarrow  \Pb(E)$ and the $\Af_5$-equivariant projection $\Pb(E)\dashrightarrow  \Pb(I)$ is given by 
\[(t_0:t_1:t_2)\mapsto (f_0(t_0,t_1,t_2):f_1(t_0,t_1,t_2):f_2(t_0,t_1,t_2))\] where
\begin{eqnarray*}
f_0&=&-x^3+y^3+z^3-\lambda x^2y-\lambda' x^2z+\lambda' xy^2-\lambda y^2z+\lambda xz^2-\lambda' yz^2,\\
f_1&=&\lambda' x^3+\lambda y^3-(\lambda+1)x^2y-(\lambda'+1)y^2x+xz^2+yz^2,\\
f_2&=&\lambda y^3+z^3-(\lambda+1)x^2y-\lambda' x^2z-\lambda y^2z+yz^2.
\end{eqnarray*}
Here $\lambda = \frac{1+\sqrt{5}}{2}$ and $\lambda' = \frac{1-\sqrt{5}}{2}$ are two roots of the quadratic equation 
$t^2-t-1 = 0$. Here we use the coordinates in the plane in which the equation of the Wiman-Edge pencil is written. Computing the Jacobian $J(f_0,f_1,f_2)$ we find the equation of a irreducible singular member of the pencil with parameters $(1:5\sqrt{5})$. 
\end{remark}

\begin{proposition}\label{prop:}
The  preimage $C_K=p^*K$ of the fundamental conic in $P$ is a nonsingular member of the Wiman-Edge pencil.  The ordered $4$-tuple  $(C_0, C_\infty, C_K, C_{K'})$ consists of four distinct members of the Wiman-Edge pencil which lie in harmonic position: there is a unique affine coordinate for $\Cs$ that identifies this $4$-tuple with $(0, \infty, 1, -1)$.
\end{proposition}
\begin{proof}
It is clear that $C_K$ is an $\Af_5$-invariant member of $|-2K_S|$ and therefore a member of the Wiman-Edge pencil. It is defined  by the quadric $Q$ 
and similarly $C_{K'}$ is defined by the quadric $Q'$. The last clause of the proposition then follows as $C_0$ (resp.\ $C_\infty$) is defined by
$Q+Q'$ (resp.\ $Q-Q'$).

If $C_K$ is singular, then it must be one of $C_c, C'_c, R, R'$. We shall exclude each of these possibilities.

Suppose that $C_K=C_c$ or $C'_c$.  Then let  $Y$  be an irreducible component of $C_K$. 
Now both $Y$ and $K$ are conics and since $p$ is a linear projection, $Y\xrightarrow{p} K$ must be of degree $1$.
On the other hand, $R$ meets $C_K$  transversally in $\Delta$, and so $Y$  meets $R$ 
transversally  in four distinct points. This implies that $Y\xrightarrow{p} K$ must have degree $3$ and we arrive at a contradiction. 

We cannot have $C_K=R$, for then $p$ would  ramify along $p^*K$ and so $p^*K$ would be $2$-divisible: this would make it twice an anticanonical 
divisor (a hyperplane section) and $R$ is clearly not of that type. If $C_K=R'$, then  $R'=C_K\xrightarrow{p} K$ must ramify in the singular part 
$\Sing(R')$ of $R'$ and so $R$ contains $\Sing(R')$. This contradicts the fact that the two members $R$ and $R'$ of the Wiman-Edge pencil intersect 
transversally. 
\end{proof}

We can also improve the statement about the birationality  of the map $(p,p')$.

\begin{theorem} The map 
\[(p,p'):S\to P\times  P'\] is a local isomorphism onto its image.  Further, 
\[f_!(1)=5(h^2+hh'+h'{}^2)\] where $(h,h')$ is the standard basis in $H^2(P\times P' ;\Z)$, and hence the composite of $f$ with
the Segre embedding 
\[P\times P=\check\Pb(I)\times\check\Pb(I')\hookrightarrow \check\Pb (I\otimes I')\cong \Pb^8\] has degree $20$.
\end{theorem}

\begin{proof}  We have already proved that the map is of degree one onto its image and hence coincides with the normalization of the image. 

It remains to show that $(p,p')$ is a local isomorphism. Suppose it is not. Since the projection $p$ resp.\ $p'$ is a local isomorphism
on $S\ssm R$ resp.\   $S\ssm R'$, the map could only fail to be local isomorphism at some $x\in R\cap R'$.   Let $L$ be the kernel of the derivative at  $x$. Then $L$ is mapped to $0$ under the composition with the both projections. This implies that $L$ coincides with the tangent line of $R$ and $R'$ at $x$. But we know that two curves in the Wiman-Edge pencil intersect transversally at a base point. This contradiction proves the assertion. 

Next represent $h$ resp.\ $h'$ by general lines $\ell\subset P$ and $\ell'\subset P'$. Then $f^*(h.h')$ is represented
by $p^*\ell .p'{}^*\ell'$. This is evidently a plane section of $S$ and hence is a class of degree $5$. Since $p$ and $p'$ are also of degree $5$, it follows that $f_*(1)$ is as asserted.
\end{proof}

\begin{remark}\label{rem:}
The map $(p,p'):S\to P\times  P'$  is not injective. We will make  this  explicit in Remark \ref{rem:4ptfibers}, but  we here give a general argument that implies
that there must  exist  a whole curve of fibers consisting of more than a single point. Let us write $f$ for $(p,p')$.
The double point formula asserts that its virtual number of double points is equal to the degree of $f^*f_!(1)-c_2(\nu_f)$ \cite{rimanyi}. The degree of 
\[f^*f_!(1)=f^*(5(h^2+hh'+h'{}^2))\] is equal to $
5\cdot 5+5\cdot 5+5\cdot 5=75$. On the other hand,  if $\mu_S\in H^4(S)$ is the fundamental class (so that $K_S^2=5\mu_S$) then 
$c(\theta_S)=1+K_S+ 7\mu_S$ (the coefficient $7$ is the topological euler characteristic  of $S$) and  \[p^*c(\theta_P)=p^*(1+3h+3h^2)=1-3K_S+15\mu_S.\]  Thus 
\[\begin{array}{ll}
c(\nu_f)&=c(\theta_S)^{-1} p^*c(\theta_P)p'{}^*c(\theta_{P'})\\
&= (1+K_S+ 7\mu_S)^{-1}(1-3K_S+15\mu_S)^{2}\\
&=(1-K_S-2\mu_S)(1-6K_S+75\mu_S)\\
&= 1-7K_S + 103\mu_S.
\end{array}
\]
This tells us that  the degree of $\nu_f$ is $103$. It follows that  the number of virtual double points is $-28$.  This can only be interpreted  this
as saying that the double point locus of $f$ must  contain a curve on $S$ with negative self-intersection.
\end{remark}

\subsection{The projection of irregular orbits} 
Let us  describe the images of the irregular $\Af_5$-orbits in $S$ under  the projection map $p:S\to P$. 
Since the projection is $\Af_5$-equivariant, the image of an irregular $\Af_5$-orbit  in $S$  is an irregular orbit in the Klein plane. 
According to Lemma \ref{invset1} there are $\Af_5$-orbits in $P$ of size $6$, $10$, $15$ (all outside the fundamental conic) and of $12$  and $20$ (all on the fundamental conic), of which those of size $6$ and $10$ come in pairs, the others being unique. 
The other irregular orbits in $P$ are of size $30$ and are parametrized by a rational curve.
On the other hand, by  Corollary \ref{cor:irrorbits}, in $S$ there are two $\Af_5$-orbits in $S$ of size $6$ and  $10$, one of size  $15$ and $20$, and 
an irreducible curve of orbits  of size $30$.

This immediately implies that the orbits on $S$ of cardinalities $6$, $10$, $15$ are mapped one-to-one to the orbits of the same cardinality in the Klein 
plane. This is also true for the $20$-element orbit, since it  consists of the base points of the Wiman-Edge pencil $\Cs$ and hence is mapped to the 
$20$-element orbit on the conic. The size $12$ orbit in $K$ must be the image of an orbit in $S$ whose size is divisible by $12$ and so this can only be a regular orbit. Since the $\Af_5$-orbits of size $30$ in $S$ and in $P$ are parametrized by an irreducible curve; $p$ will map the generic point of the former to the generic point of the latter. This information,  insofar relevant here is displayed  in Table \ref{table:orbits}; it will help us to determine the $p$-images  of the special members of the pencil. 
\vskip8mm
\tablecaption{Irregular $\Af_5$-orbits in $P$ of cardinality $\not=30$ explained by $\Af_5$-orbits in $S$.}\label{table:orbits}
\begin{center}
\begin{supertabular}{|c|| c |c|c|}
\hline special $\Af_5$-orbit $\Os$ in $S$ & $\#(\Os)$  & $p(\Os)\subset K$? & $\#(p(\Os))$\\
\hline
\hline base locus  $\Delta$ of $\Cs$ &  20 & yes & 20\\
\hline a regular orbit in $C_K$  &  60 & yes  & 12\\
\hline singular part of $C_\infty$  &  15 & no &  15\\
\hline singular part of $C_c$ or $C'_c$ &  10  (2)& no &  10\\
\hline singular part of $R$ or $R'$ &  6  (2)&  no &  6\\
\hline
\end{supertabular}
\end{center}
\vskip 5mm

\subsection{The projection of the Wiman-Edge pencil}\label{subsect:imageWP}
We will later investigate the $p$-images of the special members of the Wiman-Edge pencil, but at this point it is convenient to already make the following observation.

\begin{proposition}\label{prop:tenlineimage}
The divisor $p_*C_\infty$  is the sum of the $10$ lines that are spanned by the antipodal pairs in the 20 element orbit $p(\Delta)$ on $K$. Each singular point of  $p_*C_\infty$ lies on 
exactly two lines and the resulting $\binom{10}{2}=45$ double points make up two irregular $\Af_5$-orbits,  one of which
is the unique $15$-element orbit defined by a pairs of lines which meet in $S$ (and so the other has size $30$).
\end{proposition}
\begin{proof}
The image of a line on $S$ is a line in $P$ and so $p_*C_\infty$ is a sum of 10 lines. The polars of these lines make up an
$\Af_5$-orbit in $P$ of size $\le 10$. There is only one such orbit and it has exactly 10 elements.

The singular locus of $C_\infty$ is a 15-element orbit and we observed that  this orbit maps bijectively onto the unique 15-element orbit in $P$. It follows from  the discussion  in Subsection \ref{subsection:fixedpoints}, that the stabilizer of each singular point of $C_\infty$ is the group $\Df_4^{\ev}$. Its projection has the same stabilizer group. Thus the image of $\Sing(C_\infty)$  consists of 15 points. Since there is only one orbit in $P$ of cardinality $15$, the remaining $45-15 = 30$ points form an orbit of $\Af_5$ in $P$.
\end{proof}

This has implications  for a  \emph{generic} member $C$ of $\Cs$, as follows.  The curve $p_*C_\infty$ being  reduced and of geometric genus $6$, it follows that  $p_*C$ has the same property. As $p$ is linear, $p_*C$ is a plane curve of the same degree as $C$, namely $10$.  So the arithmetic genus of 
$p_*C$ is  $(10-1)(10-2)/2=36$, and hence its genus defect is $30$. Since the singular set of $C$ specializes to a subset 
of the  singular set of $C_\infty$, it follows that  this singular set consists of $30$ nodes and  makes up a $\Af_5$-orbit 
(but remember that such orbits  move in a curve and can degenerate into an orbit of smaller size).

So $C\in\Cs\mapsto p_*C$ defines a morphism from the base $\calB$ of the Wiman-Edge 
pencil (a copy of $\Pb^1$) to $|\Oc_P(10)|$. We denote its image  by $p_*\calB$.  It is clear that every point of $p_*\calB$ will
be an $\Af_5$-invariant curve.  An $\Af_5$-invariant curve in $P$ admits an $\Af_5$-invariant equation 
(because every  homomorphism $\Af_5\to \C^\times$ is trivial), and so every member of $p_*\Cs$ lands in the 
projectivization of $(\Sym^{10} I)^{\Af_5}$ (recall that $P$ is the projectivization of the dual of $I$).

Since $\Af_5$ acts on $I$ as the group of orientation-preserving elements of a Coxeter group of type $H_3$, this 
space is easy to determine using the invariant theory of Coxeter groups:
if $\Phi_2$ is an equation for $K$ and  $\Phi_6$, $\Phi_{10}$, $\Phi_{15}$ an equation  for the $\Af_5$-invariant union of resp.\ $6$, $10$ and $15$ lines, then these generate the $\Af_5$-invariants  in the symmetric algebra of $I$. The first three generate  the invariants of the Coxeter group and are algebraically independent ($\Phi_{15}{}^2$ is a polynomial in these). So $(\Sym^{10} I)^{\Af_5}$ is of dimension $3$ and has the basis $\{\Phi_2{}^5, \Phi_6\Phi_2{}^2, \Phi_{10}\}$. Thus every $\Af_5$-invariant decimic in $P$ can be written by equation
\beq\label{netofcurves}
a\Phi_2{}^5+b\Phi_6\Phi_2{}^2+c\Phi_{10} = 0
\eeq
In particular, $p_*\calB$ is a plane curve.
Since $p_*\Delta$ is the transversal intersection of the $10$ line union and $K$ (the common zero set of $\Phi_2$ and $\Phi_{10}$), we also see from equation \eqref{netofcurves} that  
every  $\Af_5$-invariant decimic in $P$ passes through $p_*\Delta$ and is there tangent to the $10$ line union unless it contains $K$ as an irreducible component of multiplicity $2$ (i.e., $c=0$).

The formula \eqref{netofcurves} also proves  the following

\begin{proposition}[{\bf The net of $\Af_5$-decimics}]
\label{lem:basepoints} All members of the net of $\Af_5$-decimics intersect the fundamental conic transversally at 20 points and they are all tangent at these points to one of the 10 lines that joins two antipodal points.  
\end{proposition}

Of course, the twenty points in the statement of Proposition~\ref{lem:basepoints} are the projection of the set $\Delta$ of base points of the Wiman-Edge pencil.

\begin{proposition}\label{prop:meetingconic} 
The members of $p_*\Cs$ distinct from $5K$ are reduced and intersect $K$ transversally in $p_*\Delta$.
The map $\calB\to p_*\calB$ defined by $C\mapsto p_*C$  is injective, and  $p_*\calB$ is a curve of degree $5$ in $|\Oc_P(10)|$.
\end{proposition}
\begin{proof}
If $C\in \Cs$ is not equal to $C_K$ then 
\[p_*C\cdot K=C\cdot C_K=(-2K_S)^2=20.\] Since this is also the size of 
$p(C)\cap K=p(\Delta)$, it follows that $p_*C$ is reduced and meets $K$ transversally in $p(\Delta)$.

Let $C_1, C_2$ be members of the Wiman-Edge  pencil  distinct  from $C_K$. Since  $p_*C_i$ is reduced, 
the map $p: C_i\to  p_*C_i$ is a normalization and so the equality $p_*C_1=p_*C_2$ lifts to an $\Af_5$-equivariant isomorphism  $C_1\cong C_2$. As $\Cs$ is the universal family, it follows that $C_1=C_2$. So the map $\calB\to p_*\calB$ is injective.  

It also follows that for $z\in P$ generic, then through each of the $5$ points of $p^{-1}(z)$ passes exactly one member of $\Cs$ and these members are distinct and smooth. This means that the hyperplane in $|\Oc_P(10)|$ of decimics passing through $z$ meets $p_*\calB$ transversally in $5$ points and so the curve in question has degree $5$. 
\end{proof}

In particular, $p_*\Cs$ is not a pencil. We can be a bit more precise.  If  $a\Phi_2{}^5+b\Phi_6\Phi_2{}^2+c\Phi_{10}=0$ represents  $p_*C\in p_*\Cs$, with $p_*C\not=5K$, then the fact that $p_*C$ is  transversal to $K$ implies that $c\not=0$. So such a curve has unique equation  for which $c=1$. In particular, $p_*R$ has an equation of the form 
\[\Phi_R:=a_R\Phi_2{}^5+b_R\Phi_6\Phi_2{}^2+\Phi_{10}.\] 
(We will see later that this curve is in fact the Klein decimic.) It follows that $\Phi_2, \Phi_6, \Phi_R, \Phi_{15}$ still generate the algebra $\Af_5$-invariants.

\begin{proposition}\label{prop:meetingconicsup} 
The plane curve $p_*\calB$ has two singular points, namely the points represented by $p_*R$ and by $5K$, where it has a singularity of type $A_4$ resp.\ $E_8$ (having local-analytic parameterizations  $t\mapsto (t^2, t^5)$ resp.\ $t\mapsto (t^3, t^5)$).
\end{proposition}

\begin{proof} Let $x\in \Delta$.
Since $p$ has simple ramification at $x$, we can find  local-analytic coordinates $(z_1,z_2)$ at $x$  
and $(w_1,w_2)$ at $p(x)$ such that  $p^*w_1=z_1^2$ and $p^*w_2=z_2$, and such that $K$ is at $p(x)$ given by  $w_2=0$. 
So the ramification locus $R$ is given at $x$ by $z_1=0$  and $C_K$ by $z_2=0$.  

A tangent direction  at $x$ not tangent to $C_K$
has in the $(z_1,z_2)$-coordinates  a unique generator of the form $(\lambda, 1)$. We therefore can regard $\lambda$ as 
a coordinate for the complement in  $\Pb(T_xS)$ of the point defined by $T_xC_K$ and hence as a coordinate for the complement $\calB^+\subset \calB$ of the point representing $C_K$. This means that  the  member of $\Cs\ssm \{C_K\}$ corresponding to $\lambda$ has a local parametrization at $x$ given by $z_1= \lambda z_2 (1+c_1(\lambda)z_2+\cdots )$. Its image under $p$ has then the  local  parametrization
$w_1 =\lambda^2 w_2{}^2(1+2 c_1(\lambda)w_2+\cdots)$, which shows that  $\lambda^2$, when regarded as a regular function on $\calB^+$,  is in fact a regular function on its image $p_*\calB^+$.

Now let us make these coordinate choices compatible with the chosen basis of invariants.
For this we choose a third root $\Phi_6^{1/3}$ of $\Phi_6$ and take $w_1=\Phi_R\Phi_6^{-5/3}$ (this means that $z_1$ must be a square root of this) and $w_2=\Phi_2\Phi_6^{-1/3}$. We write a member of $p_*\calB^+$ uniquely as
\[a(\lambda )\Phi_2{}^5+b(\lambda)\Phi_6\Phi_2{}^2=\Phi_R\] with $a$ and $b$ a  polynomials of degree $\le 5$ and $5$ being attained. If we multiply this equation with $\Phi_6^{-5/3}$, then this becomes \[a(\lambda)w_2{}^5+b(\lambda)w_2{}^2=w_1.\] It follows from the preceding 
that $b(\lambda)=\lambda^2$. Hence $a$ has degree $5$. By Proposition  \ref{prop:meetingconic}, the  curve $\lambda\mapsto (a(\lambda),\lambda^2)$ must be injective. This means that $a(\lambda)-a(-\lambda)$ is nonzero 
when $\lambda\not=0$. This can only happen when there is at  most one odd power of $\lambda$  appearing in $a$. This power must then be $5$, of course. It follows that  $\Pb^1\cong\calB\to p_*\calB$ is given by 
\[
[\lambda:\mu]\mapsto (\lambda^5a_5+\lambda^4\mu a_4+\lambda^2\mu^3a_2)\Phi_2{}^5 +
(\lambda^2\mu^3)\Phi_6\Phi_2{}^2-\mu^5\Phi_R,
\]
where $a_5, a_4, a_2$ are constants  with $a_5\not=0$. The proposition follows from this.
\end{proof}

\section{Images of some members of the  Wiman-Edge pencil in  the Klein plane}\label{sect:WEimages}

\subsection{The preimage of a fundamental conic}\label{subsection:preimageconic}
We will characterize $C_K$ and $C_{K'}$ as members of $\Cs$ by the fact that they support an exceptional even theta characteristic.
Recall that a \emph{theta characteristic} of  a projective smooth curve $C$ is a line bundle $\kappa$ over $C_K$ endowed with an isomorphism 
$\phi:\kappa^{\otimes 2}\cong \omega_{C}$. It is called even or odd according to the parity of the dimension of $H^0(C, \kappa)$.

The strong form of Clifford's theorem as stated in \cite{ACGH1} implies that 
\[\dim |\kappa|\le \tfrac{1}{2}(g(C)-1)\] provided
that $C$ is not hyperelliptic. If $|\kappa|$ is one-dimensional and without fixed points, then the associated morphism $C\to \check\Pb(H^0(C, \kappa))\cong \Pb^1$ 
is of degree $\le g(C)-1$ and has as fibers the moving part of $|\kappa|$. We can obtain this morphism by means of a linear projection of the canonical image:
the natural map $\Sym^2\! H^0(C, \kappa)\to  H^0(C, \omega_C)$ is injective (with image a 3-dimensional subspace), so that dually, we have a projection onto a projective plane 
\[\check\Pb(H^0(C, \omega_C))\dashrightarrow \check\Pb(\Sym^2\! H^0(C, \kappa)).\] The latter contains a conic that can be identified with the image of $\check\Pb H^0(C, \kappa)$ under the Veronese map. The composition with the canonical map $C\to \check\Pb(H^0(C, \omega_C))$ has image in this conic and thus realizes $C\to \check\Pb H^0(C, \kappa)$. Note that this conic can be understood as a quadric in $\check\Pb(H^0(C, \omega_C))$ of rank $3$ that contains the canonical image of $C$.

\begin{proposition}\label{prop:}
The morphism $C_K\to K$ is obtained as  the complete linear system  of an even theta characteristic $\kappa$ on $C_K$  followed by the (Veronese) embedding $K\subset P$. The $\Af_5$-action on $C_K$ lifts to an action of the binary icosahedral group  $\hat\Af_5$ on $\kappa$ in such a way that $H^0(C_K, \kappa)$ is an irreducible 
$\hat \Af_5$-representation of degree $2$.  Similarly for $C_{K'}\to K'$, albeit that $H^0(C_{K'}, \kappa')$ will be the other irreducible 
$\hat \Af_5$-representation of degree $2$. 

There are no other pairs $(C, \theta)$, where $C$ is a member of $\Cs$ and $\theta$ is a  $\Af_5$-invariant  theta characteristic with $\dim |\theta |=1$.
\end{proposition}
\begin{proof}
Observe that the preimage of a line in $P$ meets $C_K$ in a canonical divisor  and that any effective degree 2 divisor on $K$ spans a line in $P$. This implies that the fibers of $C_K\to K$ belong the divisor class of a theta characteristic  $(\kappa, \phi:\kappa^{\otimes 2}\cong \omega_{C_K})$. The $\Af_5$-action on $C_K$ need not lift to such an action on $\kappa$, but its central extension, the binary icosahedral group  $\bar{\Af}_5$, will (in a way that makes $\phi$ equivariant).
Thus $H^0(C_K, \kappa)$  becomes $\hat \Af_5$-representation. It contains a $2$-dimensional (base point free) subrepresentation which  accounts for the morphism  $C_K\to K$. To see that this inclusion is an equality, we note that
by Clifford's theorem as cited above, $\dim H^0(C_K, \kappa)\le 3$. If it were equal to $3$, then $\hat \Af_5$ would have  a trivial summand in $H^0(C_K, \kappa)$ and hence so would  $H^0(C_K, \kappa^{\otimes 2})=
H^0(C_K, \omega_{C_K})$. This contradicts the fact that the latter is of type $I\oplus I'$ as a $\Af_5$-representation.

It is clear that we obtain  $(C_{K'}, \kappa')$ as the transform/pull-back of $(C_K, \kappa)$ under an element of $\Sf_5\ssm \Af_5$.

If $(C, \theta)$ is as in the proposition, then
$C$ lies on a $\Af_5$-invariant quadric of rank $3$. This quadric will be defined by $\lambda Q+\lambda' Q'$ for some $(\lambda:\lambda')$ (with $Q$ and $Q'$ as in Subsection \ref{subsect:acmodel}), and so the rank condition implies $\lambda=0$ or $\lambda'=0$. In other words, $(C, \theta)$ equals $(C_K,\kappa)$ or $(C_{K'}, \kappa')$. 
\end{proof}

\subsection{The image of reducible singular members}
Recall that $C_c$ and $C'_c$ is the $\Af_5$-orbit of a special conic. So the   following proposition tells us what their $p$-images are like.

\begin{proposition}\label{prop:}
The projection  $p$ maps each special conic  isomorphically onto a conic in $P$ that is tangent to four lines that are the projections of lines on $S$, and  passes through four points of the 10-element $\Af_5$-orbit. 
\end{proposition}
\begin{proof}
Let $Y$ be a special  conic. Since $Y$ is part of $C_c$ or $C'_c$, $Y$ meets $\Delta$ in $4$ distinct points and  $p_*Y$ is a degree 2 curve which meets $p(\Delta)$ in $4$ points. In particular $p$ maps $Y$ isomorphically onto its image (not on a double line). For each $a\in Y\cap \Delta$, $p(a)$ and its antipode  span a line in $P$ and according to Proposition \ref{prop:meetingconic}  $p(Y)$ is tangent to the line at $p(a)$.  
\end{proof} 

\subsection{The  decimics of Klein and Winger}\label{subsect:winger}
Let us call a projective plane  endowed with a group $G$ of automorphisms  isomorphic to $\Af_5$ a \emph{Klein plane} (but without specifying an isomorphism $G\cong \Af_5$). Such a plane is unique up to isomorphism (it is isomorphic to both $P$ and $P'$).
For what follows it is convenient to make the  following definition.

\begin{definition}[{\bf Special decimics}]
\label{def:}
We call a reduced, $G$-invariant curve in a Klein plane of degree $10$ a \emph{special decimic} if is singular at each fundamental point
and its normalization is rational. If the singularity at a fundamental point is an ordinary node, we call it a \emph{Winger decimic}; otherwise
(so when it is worse than that) we call it a \emph{Klein decimic}.
\end{definition}

We will see that each of these decimics is unique up to isomorphism, and that $p(R)$ is a Klein decimic and $p(R')$ a Winger decimic.
We will also show that the singularity of a Klein decimic at a fundamental point must be a double cusp (i.e., with local analytic equation $(x^3-y^2)(y^3-x^3)$).
As to our naming:  a Klein curve appears in \cite{Klein} , Ch.\ 4, \S  3 (p.\ 218 in the cited edition), and a Winger curve appears in \S 9 of \cite{Winger}.

Let us first establish the relation between special decimics and the Wiman-Edge pencil.
Let $K$ be a copy of a $\Pb^1$ and let $G\subset \Aut(K)$ a subgroup isomorphic to $\Af_5$.
Recall that  $K$ has  a unique $G$-orbit $F^\#$ of size $12$ (think of this as the vertex set of a regular icosahedron) which comes in  $6$ antipodal pairs. 
Let us denote  by $z\in F^\#\mapsto z'\in F^\#$ the antipodal involution, and let 
$\overline K$ be obtained from $\tilde Y$ by identifying every $z\in F^\#$ with $z'\in F^\#$ as to produce an ordinary node. 
This is just the curve of arithmetic genus $6$ that we constructed in Lemma \ref{lemma:glueingconstruct}. 

Notice that the normalization of a special decimic factors through this quotient of $K$: it is the image of $\overline K$  under a $2$-dimensional linear system 
of degree $10$ divisors on $\overline K$ that comes from a $3$-dimensional irreducible subrepresentation of $H^0(\overline K, \Oc_{\overline K}(10))$.
In order to identify these subrepresentations, we focus our attention on  the dualizing sheaf $\omega_{\overline K}$ of $\overline K$.
This is the subsheaf of the direct image of $\omega_K(F^\#)$ characterized by the property that the sum of the residues in a fiber add up to zero. So
$H^0(\overline K, \omega_{\overline K})$ is the subspace of $H^0(K,\omega_K(F^\#))$ consisting of differentials whose  residues at the two
points of any antipodal pair add up to zero.  

\begin{lemma}\label{lemma:irrsplitoff5}
There is an exact sequence of $G$-representations
\[
0\to H^0(\overline K, \omega_{\overline K})\to H^0(K, \omega_K(F^\#))\to \tilde H^0(F; \C)\to 0
\]
where the last term denotes the reduced cohomology of $F$.  Further,  $\tilde H^0(F; \C)$ is an irreducible $G$-representation of dimension $5$ (hence of type $W$).
\end{lemma}
\begin{proof}
For every  $x\in F$ there is a linear  form on $H^0(K, \omega_K(F^\#))$ that assigns to $\alpha\in H^0K, \omega_K(F^\#))$ the sum of the residues
at the associated antipodal pair on $K$. By the residue formula, these $6$ linear forms add up to zero. Apart from that, residues can be arbitrarily 
described  and we thus obtain the exact sequence. The character of the permutation representation on $F$ (which can think of as the set of $6$ lines through opposite vertices of the  icosahedron) is computed to be that of trivial representation plus that of $W$. This implies that $\tilde H^0(F; \C)\cong W$ as $G$-representations.
\end{proof}

We know by Proposition \ref{prop:irred} and  Proposition \ref{prop:Kleinram}  that the  dualizing sheaf of $\overline K$ defines the canonical embedding for $\overline K$ and that  the  $6$-dimensional $G$-representation  $H^0(\overline K, \omega_{\overline K})$ decomposes into two irreducible subrepresentations of dimension $3$ that are not of the same type. This enables us to prove that there are only two isomorphism types of special decimics.

\begin{corollary}[{\bf Classification of special decimics}]
\label{cor:}
Let $Y\subset P$ be a  special decimic and let $\overline K\to Y$ define a partial $G$-equivariant normalization (so that  the $6$ nodes of $\overline K$ lie over the  $6$ fundamental points of $P$). Then $Y$ can be identified in a $G$-equivariant manner with the image of $\overline K$ under the linear system  that comes from one of the two irreducible $3$-dimensional $G$-subrepresentations of $H^0(\overline K, \omega_{\overline K})$. 
In fact,  $p_*R$ and $p'_*R$ are special decimics and every special decimic is isomorphic to one of them.
\end{corollary}
\begin{proof}
The line bundle $\omega_K(F^\#)$ is of degree $10$. The fact that $\omega_{\overline K}$ is of degree $-2+12=10$ implies that $\Pb(H^0(K, \omega_K(F^\#)))$ is the complete  linear system of degree $10$. As the $G$-embedding $Y\subset P$ is of degree $10$, the $G$-embedding $Y\subset P$ is definable by a 
$3$-dimensional $G$-invariant subspace of $H^0(K, \omega_K(F^\#))$. It follows from  Lemma \ref{lemma:irrsplitoff5} that this subrepresentation 
must be contained in $H^0(\overline K, \omega_{\overline K})$, and hence is given by one of its $3$-dimensional summands. 

Let us write $E_K$ for $H^0(\overline K, \omega_{\overline K})$ and $\Pb_K$ for $ \check\Pb(E_K)$.
The canonical map $\overline K\to  \Pb_K$ is an embedding and realizes  $\overline K$ as a member of the Wiman-Edge pencil: if  we regard this embedding as an inclusion, then  the $\Af_5$-embedding  $R\subset \Pb_S$ is obtained from the $G$-embedding  $\overline K\subset \Pb_K$ via a compatible pair of isomorphisms $(G, E_K)\cong (\Af_5, E_S)$.  The preimage $S_K\subset \Pb_K$ of $S\subset \Pb_S$ is a $G$-invariant quintic del Pezzo surface  which contains $\overline K$.  We have a decomposition $E_K=I_K\oplus I'_K$  into two irreducible subrepresentations such that the two associated linear systems of dimension $2$ reproduce the projections $p|R$ and $p'|R$. 
So $p_*R$ and $p'_*R$ represent the two types of special decimics.
\end{proof}

\begin{theorem}\label{thm:6-orbits}
The  preimage $p^{-1}F$ of the set of fundamental points is the disjoint union of $\Sing(R)$ and $\Sing(R')$. 
Moreover, $p_*R'$ is a Winger decimic  and $p_*R$ is  a Klein decimic. 
\end{theorem}
\begin{proof}
We know that the singular set of $R$ makes up an $\Af_5$-orbit and that such an $\Af_5$-orbit is mapped to $F$. The same is true for $R'$ and so each fundamental point is a singular point of both $p_*R$ and $p_*R'$.   For every  singular point $x$ of $R$,  $p^{-1}p(x)\ssm \{x\}$ consists of $k\le 5-2= 3$ points. So $p^{-1}p(\Sing(R))\ssm \Sing(R)$ is a  $\Af_5$-invariant set of size $6k \in \{6,12,18\}$. 

Our irregular orbit catalogue for the $\Af_5$ action on $S$ (see \S\ref{subsection:fixedpoints}) shows that only $k = 1$ is possible, so that this must be the $6$-element orbit in $S$ different from $\Sing(R)$, i.e.,  
$\Sing(R')$. This proves that the (disjoint) union of  $\Sing(R)$ and $\Sing(R')$ make up a fiber of $p$.  In particular,  $p_*R'$ has an ordinary double point at $p(x)$ and hence is a Winger decimic. On the other hand, $p_*R$ has multiplicity at least $4$ at $p(x)$: this is because the restriction of $p$ to a local branch of $R$ has a singularity at $x$ (because of the presence of the other branch). So $p_*R$ must be a Klein decimic.
\end{proof}

\begin{remark}\label{rem:4ptfibers}
Theorem \ref{thm:6-orbits} shows that there is a unique $\Af_5$-equivariant  bijection 
\[\Sing(C_{\ir})\xrightarrow{\cong}\Sing(C_{\ir}')\] that commutes with $p$. This gives rise to a bijection $\sigma: F\cong F'$ between the set of fundamental points in $P$ and those in $P'$.
Similarly, $p'$ will determine a $\Af_5$-equivariant bijection $\sigma': F'\to F$. The composition $\sigma'\circ\sigma$ is a permutation of $F$  that commutes with the $\Af_5$ action. Since that $\Af_5$ action contains $5$-cycles (which have just one fixed point in  $F$), it follows that this permutation must be the identity: the  two bijections are each others inverse. It follows  that for every $x\in \Sing(C_{\ir})$, the set $\{x, \sigma(x)\}$ is a fiber of 
$p\times p' :S\to P\times P'$.\end{remark}

We can now also say a bit more about the special decimics.

\begin{proposition}\label{prop:}
The singular set of a  Winger decimic consists of 36 ordinary double points, and a  local branch at each fundamental point  has that node as a hyperflex.
The singular points that are not fundamental  make up a $30$-element orbit. 
\end{proposition}
\begin{proof}
For the first assertion, we essentially follow  Winger's argument. The normalization $q: \tilde Y\to Y\subset P$ is a rational curve that comes with an automorphism group $G\cong \Af_5$. The points of  $\tilde Y$ mapping to a flex point or worse make up a divisor $D$  of degree $3(10-2)=24$ (if $t$ is an affine coordinate of $\tilde Y$, and $(z_0:z_1:z_2)$ is a coordinate system for $P$ such that $q^*z_0$ resp.\  $q^*z_i$ is a polynomial  of degree $10-i$ and
resp. $\le 10-i$, then $D$ is the defined by the Wronskian determinant  $(z_0,z_1,z_2)$  and is viewed as having degree $(10+9+8)-3=24$).  If a point of $\tilde Y$ maps to cusp or a hyperflex (or worse) then it appears with multiplicity $\ge 2$. We prove that each of the two points of $\tilde Y$ lying over a fundamental point  of $P$ has multiplicity $\ge 2$ in $D$. 
This implies that $D$ is $\ge$ twice the $12$-element $G$-orbit in $\tilde Y$ and hence, in view of its degree, must be equal to this. It will then follow that
$Y$ has only nodal singularities and that these will  be $36$ in number by the  genus formula.

Let us first note that since  $G$ has no nontrivial $1$-dimensional character, $Y$ admits a $G$-invariant equation.
Let $x\in P$ be a fundamental point.  Let $U\subset P$ be the affine plane complementary to the polar of $x$ and make it a vector space by choosing $x$ as its origin. Then  $G_x$ acts linearly on $U$. The stabilizer  $G_x$ is a dihedral group of order $10$, and we can choose coordinates $(w_1,w_2)$ such that $(w_1, w_2)\mapsto (w_2,w_1)$ and  $(w_1, w_2)\mapsto (\zeta_5w_1,\zeta_5^{-1}w_2)$ define generators of $G_x$. The algebra of $G_x$-invariant polynomials is then generated by $w_1w_2$ and $w_1{}^5+w_2{}^5$.  

The germ of 
$Y$ at  $x$ admits  $G_x$-invariant equation and this is given by $w_1w_2+\text{order}\ge 2$. So the tangent lines of the branches of $Y$ at $x$ are the coordinate lines. The subgroup of $G_x$ generated by $(w_1, w_2)\mapsto (\zeta_5w_1,\zeta_5^{-1}w_2)$ fixes each branch. So the branch whose tangent line is $w_1=0$ has a local equation of the type 
\[w_1+aw_1{}^2w_2 +bw_1{}^3w_2{}^2 +cw_2^4+ (\text{terms of order $\ge 5$})=0.\] This branch meets $w_1=0$ 
at the origin with multiplicity $\ge 4$ and hence has  there a hyperflex. 

So $\Sing(p_*R')\ssm F$ is a $\Af_5$-invariant a $30$-element subset of $P$ and it suffices to prove that this is a $\Af_5$-orbit.
This set does not meet the fundamental conic $K$ (for  $p_*R'$ is  transversal to that conic). There are unique $\Af_5$-orbits in $P\ssm K$ 
of size $6$ and $10$; the others have size  $30$ and $60$. Hence $\Sing(p_*R')\ssm F$ is a $30$-element orbit in $P\ssm K$.
\end{proof}

\subsection{Construction of the Klein decimic} We now give a construction of the Klein decimic, which will at the same time show that it will have double cusp singularities at the fundamental points. 

Let $P$ be a Klein plane. We first observe that its  fundamental set  $F$ does not lie on a conic. For if it did, then this conic would be unique,  hence $\Af_5$-invariant and so its intersection with $K$ would then produce a orbit with $\le 4$ elements, which we know does not exist. This implies that for each $x\in F$ there is a conic $K_x$ which meets $F$ in $F\ssm\{ x\}$. 

Now blow up $F$. The result is a cubic surface $X$ for which the strict transform $\tilde K_x$ of $K_x$ is a line 
(an exceptional  curve) and these lines are pairwise disjoint(\footnote{This surface is isomorphic to the \emph{Clebsch diagonal cubic surface} in $\Pb^4$ defined by $\sum_{i=0}^4 Z_i^3=0$, $\sum_{i=0}^4 Z_i=0$;  the evident $\Sf_5$-symmetry accounts for its  full automorphism group and its isomorphism type is characterized  by that property. The intersection of this cubic surface with the quadric defined by 
$\sum_{i=0}^4 Z_i^2=0$ is the \emph{Bring curve} mentioned in Remark \ref{rem:wingerpencil}; its automorphism group is also $\Sf_5$.}).  The set of $6$ exceptional curves arising from the blowup and 
the $\tilde K_x$ form what is called a \emph{double six}  in  $X$. So they can be simultaneously blown down to form a copy $P^\dagger$ of  a projective plane. The naturality of this construction 
implies that the $\Af_5$-symmetry is preserved. In particular, $P^\dagger$ comes with a nontrivial $\Af_5$-action, and so we have defined a fundamental conic $K^\dagger\subset P^\dagger$.  The  image of $\cup_x\tilde K_x$ in $P^\dagger$ is a $\Af_5$-invariant subset of size $6$ and so this must be the fundamental set $F^\dagger\subset P^\dagger$. So the diagram $P\xleftarrow{q}X\xrightarrow{q^\dagger} P^\dagger$ is involutive.  The birational map 
\[(q^{\dagger})^{-1}q: P\dashrightarrow P^\dagger\] is defined by the linear system $I_K$ of quintics in $P$ which have a node (or worse) at each fundamental point, so that $P^\dagger$ gets identified with the projective space dual to $|I_K|$.

\begin{lemma}\label{lemma:tangentconic}
If  $\ell_x$ is the polar line of $x$, then  $K_x\cap K=\ell_x\cap K$ and hence $K_x$ is tangent to $K$ at these two points. 
\end{lemma}
\begin{proof}
The divisor $\sum_{x\in F} K_x\cdot K$ in $K$ is $\Af_5$-invariant and  of degree $6\cdot 4=24$. Given the $\Af_5$-orbit sizes in $K$ ($12, 20, 30, 60$), this implies that that this divisor is twice the $12$-element orbit. Hence $K_x$ meets $K$ in 
two distinct points of this orbit and is there tangent to $K$. Assigning to $x$ the polar of $K\cap K_x$ then gives us a map
from $F$ onto a $6$-element orbit of $P$. This orbit can only be $F$ itself. We thus obtain a permutation of $F$ which commutes with the $\Af_5$-action. Since every point of $F$ is characterized by its $\Af_5$-stabilizer, this permutation must be the identity.
\end{proof}

\begin{corollary}\label{cor:}
The Cremona transformation $q^\dagger q^{-1}: P\dashrightarrow P^\dagger$ takes $K\subset P$ to a Klein decimic in $P^\dagger$. 
This Klein decimic has a double cusp (with local-analytic equation $(x^2-y^3)(x^3-y^2)=0$)
at each fundamental point, and is smooth elsewhere.
\end{corollary}
\begin{proof}
Since $q$ is an isomorphism over $K$, we can identify $K$ with $q^{-1}K$. Any singular point of  $q^\dagger q^{-1}K$ will
of course be the image of some $K_x$.
Since we can regard $K_x$ as the projectivized tangent space of the fundamental point in $P^\dagger$ to which it maps, 
Lemma \ref{lemma:tangentconic} shows that the image of $K$ at this fundamental point is as asserted: we have 
two cusps meeting there with different tangent lines. 
\end{proof}

Note that by the involutive nature of this construction, $q(q^\dagger)^{-1}$ will take $K^\dagger$ to the Klein decimic
in $P$. 

\begin{remark}\label{rem:doublecuspdiscriminant}
Theorem \ref{thm:6-orbits} and the preceding corollary  imply that the double cusp singularity appears as a discriminant curve of a finite morphism between surface 
germs of degree $4$. Here is a local description  for it that also takes into account the $\Df_5$-symmetry: there exist a
local-analytic coordinate system  $(z_1,z_2)$ for $S$ at a singular point of $R$ and  a local-analytic coordinate system  $(w_1,w_2)$ 
at its image in $P$ such that
\[
p^*w_1= z_1^2+z_2^3,\;  p^*w_2=z_1^3 +z_2^2.
\]
Note that it is indeed of degree $4$ at the origin. The ramification locus is defined $ z_1 z_2(9 z_1 z_2-4)=0$, hence is near the origin given by $ z_1 z_2=0$.
The  image of the ramification locus is at the origin is the double cusp, for the $ z_1$-axis,  parametrized by  $(z_1,z_2)=(t,0)$ is the 
the parametrized cusp $(w_1,w_2) =(t^2, t^3)$ and likewise $z_2$-axis  maps to the  parametrized cusp $(w_1,w_2) =(t^3, t^2)$.

The $\Af_5$-isotropy groups of the two points is a dihedral isotropy group of order $10$ and so the map-germ must have this symmetry as well. 
We can see this being realized  here as in fact coming from two linear representations of this group. Let $\Df_5$ be the  dihedral group of order $10$, 
thought of as the semi-direct product of the group $\mu_5$ of the $5$th roots of unity and  an order two  group  whose nontrivial element $s$ acts by inversion: $s\zeta =\zeta^{-1}s$. Then letting  $\Df_5$ act on the source resp.\ target by
\begin{align*}
 s(z_1,z_2)=(z_2,z_1);\; &  \zeta (z_1,z_2)=( \zeta z_1,  \zeta^{-1}z_2)\\
      s(w_1,w_2)=(w_2,w_1);\;  &  \zeta (w_1,w_2)=( \zeta^{2}w_1,  \zeta^{-2} w_2).
\end{align*}
makes the map-germ $\Df_5$-equivariant. It can also be verified that  the $\Sigma^{1,1}$-multiplicity of this germ is $3$, as it must be, in view of the 
proof of Proposition \ref{prop:Kleinram}.
\end{remark}

\begin{remark}\label{rem:wingerpencil}
We have seen that the curve of Wiman decimics in $P$ contains both $5K$ and the $\Af_5$-invariant sum of $10$ lines. It has as a remarkable counterpart: 
the pencil of sextics spanned by  $3K$ and the $\Af_5$-invariant sum of $6$ lines. This pencil  was studied in detail by R.~Winger  \cite{Winger} and is discussed in \cite{CAG}, Remark 9.5.11. The $6$ lines meet $K$ in the $12$-element orbit and this intersection is evidently transversal.
It follows that this orbit is the base locus of the \emph{Winger pencil} (as we will call it)  and that all  members of the pencil
except $3K$ have a base point as flex point (with tangent line the fundamental line passing through it). Note that a  general member $W$  of the pencil is smooth of genus $10$. 

Winger shows that this pencil has, besides its two generators, which  are evidently singular, two other singular members. Both are irreducible;  one of them has as its singular locus a node at each fundamental point (the $6$-element  orbit of $\Af_5$) and 
the other, which we shall call the \emph{Winger sextic}, has as its singular locus a node at each point of the $10$-element  orbit of $\Af_5$, and each local branch at such a point has this point as a flex point. So their normalizations have genus $4$ and $0$ respectively. The former turns out be  isomorphic to the Bring curve (whose automorphism group is known to 
be isomorphic to $\Sf_5$) and the latter will be isomorphic as an $\Af_5$-curve to $K$. 

In fact, as an abstract curve with $\Af_5$-action, the Winger sextic  can be obtained in much the same way as the curve constructed in Lemma \ref{lemma:glueingconstruct}, simply by identifying the antipodal pairs of its $20$-element orbit so as to produce a curve with $10$ nodes. This plane sextic is also discussed in \cite{DolgQuartics}.  With the help of the  Pl\"ucker formula (e.g., \cite{CAG}, formula 1.50), it then follows that the dual curve  of the Winger sextic is a  Klein decimic. (A priori this curve lies in the dual of $P$, but in  the presence  of the fundamental conic $K\subset P$ we can regard it as a curve in $P$: just assign  to each point of the Winger curve the $K$-polar of its tangent line.) 

Other remarkable members of this pencil include the two  nonsingular \emph{Valentiner sextics}\index{Valentiner sextic}  with automorphism group isomorphic to $\Af_6$ (\footnote{The first author uses the opportunity to correct the statement in \cite{CAG}, Remark 9.5.11 where the Valentiner curve was incorrectly identified with the Wiman sextic.}).
\end{remark}

\section{Passage to $\Sf_5$-orbit spaces}\label{sect:orbitspace}

The goal of this section to characterize the Wiman curve as a curve on $\calMc_{0,5}$, or rather its $\Sf_5$-orbit space in $\Sf_5\bs\calMc_{0,5}$.
The latter is simply the moduli space of stable effective divisors of degree $5$ on $\Pb^1$.

\subsection{Conical structure of  $\Sf_5\bs \calMc_{0,5}$}\label{subsect:conical} 
We  first want to understand how the $\Sf_5$-stabilizer $\Sf_{5, z}$ of a point $z\in \calMc_{0,5}$ acts near $z$. The finite group action on the complex-analytic germ of $\calMc_{0,5}$ at $z$ can be linearized in the sense  that it is complex-analytically equivalent  to the action on the tangent space. It then follows from a theorem of  Chevalley that the local orbit space at $\Sf_5\bs \calMc_{0,5}$ is nonsingular at the image of $z$ if and only if $\Sf_{5, z}$ acts on $T_z\calMc_{0,5}$ as a complex reflection group.

\begin{lemma}\label{lemma:coxeter}
The $\Sf_5$-stabilizer of any point not in $\Delta$ is a Coxeter  group (and hence a complex reflection group) so that $\Sf_5\bs \calMc_{0,5}$ is nonsingular away from the point $\delta$ that represents the orbit $\Delta$. 
\end{lemma}
\begin{proof}
If $z$ is represented in the Hilbert-Mumford model by $(x_1,\dots, x_5)\in (\Pb^1)^5$, then the  tangent space $T_z\calMc_{0,5}$ fits  naturally in a short exact sequence:
\begin{equation}\label{eq:exactsequence}
0\to H^0(\Pb^1, \theta_{\Pb^1})\to \oplus_{i=1}^5T_{x_i}\Pb^1\to T_z\calMc_{0,5}\to 0
\end{equation}
where  $H^0(\Pb^1, \theta_{\Pb^1})$  is the space of vector fields on $\Pb^1$ (which in terms of an affine coordinate $t$ has basis $\frac{\partial}{\partial t}, t\frac{\partial}{\partial t}, t^2\frac{\partial}{\partial t}$) and the map is given by evaluation at $x_1, \dots, x_5$ respectively. With the help of this formula it is fairly straightforward to determine the action of $\Sf_{5, z}$ on $T_z\calMc_{0,5}$. When  $\Sf_{5, z}$ is of order two, we need to verify that $\Sf_{5, z}$ leaves a line in $T_z\calMc_{0,5}$ pointwise fixed.
This is left to the reader. There are three  other cases to consider:

Case 1: $z=(1, \zeta_3, \zeta_3^2, 0, \infty)$, and $\Sf_{5,z}\cong \Sf_3^{\ev}$ is the group generated by $t\mapsto \zeta_3t$ and the involution $t\mapsto t^{-1}$
(which gives  the exchanges  $\zeta_3\leftrightarrow \zeta_3^2$  and $0\leftrightarrow\infty$).  
The fact that $H^0(\Pb^1, \theta_{\Pb^1})$ maps onto the direct sum of tangent spaces at the third root of unity implies that 
\[T_0\Pb^1\oplus T_\infty\Pb^1\to T_z\calMc_{0,5}\] is an isomorphism. Multiplication by $\zeta_3$ acts here as $(v_0, v_\infty)\mapsto  (\zeta_3v_0, \zeta_3^{-1}v_\infty)$ and inversion exchanges $v_0$ and $v_\infty$. We thus find that $\Sf_{5,z}$ acts as the Coxeter group $A_2$ on $T_z\calMc_{0,5}$.
\smallskip

Case 2: $z=(0,0,1,\infty, \infty)$, and $\Sf_{5,z}\cong \Df_8$ acts in the obvious manner. We again find that $T_0\Pb^1\oplus T_\infty\Pb^1\to T_z\calMc_{0,5}$ is an isomorphism. The group  $\Sf_{5,z}$  now acts as the Coxeter group $B_2$: It consists of the signed permutations of $(v_0, v_\infty)$.
\smallskip

Case 3: $z=(1,\zeta_5, \zeta_5^2, \zeta_5^3, \zeta_5^4)$, and $\Sf_{5,z}\cong \Df_{10}$ is generated by  $t\mapsto \zeta_5t$ and the involution $t\mapsto t^{-1}$ (which gives  the exchanges  $\zeta_5^i\leftrightarrow \zeta_5^{-i}$). For $i\ge -1$ let $v^{(i)}\in \oplus_{\zeta^5=1}T_{\zeta}\Pb^1$  be the restriction of the vector field $t^{i+1}\frac{d}{dt}$. Then scalar multiplication by $\zeta_5$ multiplies  $v^{(i)}$ by $\zeta_5^{-i}$, so that this is a basis of eigenvectors for this action. The vectors  $v^{(-1)}$,  $v^{(0)}$,  $v^{(1)}$ span  the image of $H^0(\Pb^1, \theta_{\Pb^1})$, and hence $v^{(2)}$ and $v^{(3)}$ span $T_z\calMc_{0,5}$.
Multiplication by $\zeta_5$ sends $(v^{(2)},v^{(3)})$ to $(\zeta_5^{-2}v^{(2)},\zeta_5^{-3}v^{(3)})$ and the involution sends it to  $(-v^{(3)},-v^{(2)})$.  Thus it acts through the Coxeter group $I_2(5)$.
\end{proof}

It is clear from Proposition \ref{prop:orbifolds} that the $\Af_5$-orbit space 
of every smooth member of $\Cs$ is a smooth rational curve. This is also true for a singular member, for  we have seen that  $\Af_5$ is transitive on its irreducible 
components and that an irreducible component is rational (a quotient curve of a rational curve is rational). It is clear  that $\Sf_5\bs \calMc_{0,5}$ will inherit  this structure. Since the base locus $\Delta$ is a single  $\Af_5$-orbit and all members of the Wiman-Edge pencil  are nonsingular at $\Delta$, it then follows that
the image of the  Wiman-Edge pencil gives $\Af_5\bs \calMc_{0,5}$ the structure of a quasi-cone with vertex $\delta$.

\begin{proposition}\label{prop:cone}
Regard $\Pb^2$ as a projective cone with vertex the origin $o$ of the affine part $\C^2\subset\Pb^2$.  Let $\Cf_6$  be the automorphism group of this cone generated by 
\[(u_0:u_1:u_2)\mapsto (u_0:\zeta_3 u_1,-\zeta_3 u_2).\] Then $\Sf_5\bs \calMc_{0,5}$  is as a quasi-cone  isomorphic to $\Cf_6\bs\Pb^2$ (which has an $A_2$-singularity at its vertex) and embeds in a projective space by means of the $5$-dimensional linear system  defined by the degree 6 subspace of $\C[u_0,u_1, u_2]^{\Cf_6}$.  This embedding is of degree $6$  and is such that  $C_0$ maps onto a line; any other member of $\Cs$ maps onto a conic.
\end{proposition}
\begin{proof}
The type of this cone can be determined by a local study at a point of $\Delta$ in the way we did this in the proof of Lemma \ref{lemma:coxeter}. One such point is represented in the Hilbert-Mumford model by $z:=(0, 0, 1, \zeta_3 ,\zeta_3^2)\in (\Pb^1)^5$. The $\Sf_5$-stabilizer of $z$ is cyclic of order $6$ with a generator $\sigma$ with
\[\sigma(x_1,x_2, x_3, x_4, x_5)=(x_2,x_1, \zeta_3 x_3, \zeta_3 x_4, \zeta_3 x_5)\] and the tangent space $T_z\calMc_{0,5}$ is  
identified with $T_0\Pb^1\oplus T_0\Pb^1$. If we identify $T_0\Pb^1$ with $\C$ by taking the coefficient of
$\frac{\partial}{\partial t}\vert_{t=0}$, then we see that  $\sigma$ acts as $(t_1, t_2)\in \C^2\mapsto (\zeta_3 t_2, \zeta_3 t_1)\in \C^2$. So if we pass to
$(u_1,u_2)=(t_1+t_2, t_1-t_2)$, then $\sigma (u_1,u_2)=(\zeta_3 u_1,-\zeta_3 u_2)$.  An affine  neighborhood of the vertex of the cone $\Sf_5\bs \calMc_{0,5}$ is the quotient  of $\C^2$ by the automorphism 
group generated by $\sigma$.  

The diagonal form  of $\sigma$ shows that it leaves invariant only two lines through the origin: the 
$u_2$-axis defined by $u_1=0$, on which it acts faithfully with order $6$, and the $u_1$-axis defined by $u_2=0$ on which it acts with order $3$.
Since $C_0$ and $C_\infty$ are the only members of $\Cs$ that are $\Sf_5$-invariant, these axes must define their germs at $z$. The $\Sf_5$-stabilizer of $z$ (which we have identified with $\Cf_6$) acts  on the germ of $C_\infty$ at $0$ through a subgroup of $\Sf_3$ (for its preserves a line on the quintic del Pezzo surface) and so $\Cf_6$ cannot act faithfully on it. It follows that $C_0$ is defined by 
 $u_1=0$.

Since  $\sigma^3$ sends $(u_1,u_2)$ to $(u_1, -u_2)$, its algebra of invariants is $\C[u_1, u_2^2]$. 
Now $\sigma$ sends $(u_1, u_2^2)$ to $(\zeta_3 u_1, \zeta_3^{-1}u_2^2)$,
so that its orbit space produces a Kleinian singularity of type $A_2$.  Following Klein, its algebra of invariants is generated by 
$u_1^3, u_1u_2^2, u_2^6$.  So the degree $6$ part of $\C[u_0,u_1, u_2]^{\Cf_6}$ has basis $u_0^6$, $u_0^3u_1^3$, $u_0^3u_1u_2^2$, $u_2^6$,
$u_1^6$, $u_1^2u_2^4$, $u_1^4u_2^2$,  and this $\Cf_6\bs\Pb^2$ in $\Pb^5$.
Its image is clearly of degree $6$. We have seen that the image of $C_0$ is given by  $u_1=0$ and the above basis restricted to nonzero monomials are  $u_0^6$ and $u_2^6$ and so its image is a line.
If we substitute $u_2=\lambda u_1$ (with $\lambda$ fixed),  then the  monomials in question span $u_0^6,u_0^3u_1^3, u_1^6$ and so the image
of this line is indeed a conic. 
\end{proof}

\begin{remark} The conclusion of the previous Lemma and Proposition agrees with the explicit isomorphism $\Sf_5\bs \calMc_{0,5}\cong \Pb(1,2,3)$  from Remark \ref{rem:clebsch} below. The only singular point of the quotient is the point $(0:0:1)$, which is a rational double  point of type $A_2$. The points projecting to the singular point satisfy $I_4=I_8 = 0$. We know that $I_4 = 0$ defines the Wiman curve on $S$ and $I_8 = 0$ defines the union $C_K+C_K'$. They intersect at 20 base points with the subgroup $\Sf_3^{\odd}$ as the stabilizer subgroup. 

The weighted projective space $\Pb(1,2,3)$ is obtained as the quotient of $\Pb^2$ by a cyclic group of order 6 acting on coordinates as \[(x_0:x_1:x_2)\mapsto (x_0:\zeta_6^3 x_1:\zeta_6^2x_2).\] The linear space of $\Cf_3$-invariant homogeneous polynomials of degree 6 is spanned by 6 monomials $x_0^6,x_1^6,x_0^4x_1^2,x_0^2x_1^4,x_2^6,x_0^3x_2^3$.  They define a $\Cf_6$-invariant map $\Pb^2\to \Pb^5$ which factors through the quotient $\Pb(1,2,3) = \Pb^2/\Cf_6$. In terms of weighted homogeneous coordinates $t_0,t_1,t_2$ in $\Pb(1,2,3)$ it is given by monomials of the map 
$t_0^6,t_1^3,t_0^4t_1,t_0^2t_1^2,t_2^2,t_0^3t_2$. The restriction of the map to the coordinate line $t_0 = 0$ is given by the monomials $t_1^3,t_2^2$. This shows that the image of the Wiman curve under the composition 
\[S = \calM_{0,5}\to \Pb(1,2,3)\to \Pb^5\] is a line.

One can interpret this map as the map 
\[f:\calM_{0,5}\to \Pb(V(5)) = \Sf_5\backslash (\Pb^1)^5\cong \Pb^5,\] where $V(5)$ is the space of binary quintic. According to \cite{Abdesselam}, this map can be given explicitly by 
$$(a_1,\ldots,a_5)\mapsto R(a_1,\ldots,a_5) = \prod_{i=1}^5(T_iz-S_i^3w) = \sum_{i=0}^5B_iz^{5-i}w^i,$$
where $S,T$ are fundamental invariants of degrees $2$ and $3$ of binary forms of degree $4$ satisfying $S^3-27T^2$ equal to the discriminant. 
The reference gives explicit expressions for the coefficients $B_i$ as polynomials of degree 24 in basic invariants $I_4,I_8,I_{12}$ of  binary quintic forms. 
\end{remark}

\subsection{The cross  ratio map}
The embedding of $\Sf_5\bs\calMc_{0,5}$ in projective $5$-space that we found in Proposition \ref{prop:cone} has a modular interpretation, which we now give.

As was mentioned in Subsection \ref{subsect:generalities}, the conic bundles of $S$  have such an interpretation on $\calMc_{0,5}$
as forgetful morphisms \[f_i:\calMc_{0,5}\to \calMc_{0,4}{}^{(i)} \cong \Pb^1\] 
(it forgets the $i$th point). The notation 
$\calMc_{0,4}{}^{(i)}$ used here means that we do not  renumber the four points and let  them be indexed by $\{1,\dots, 5\}\ssm\{i\}$; this helps to make manifest
the $\Sf_5$-action on the product $\prod_{i=1}^5\calMc_{0,4}{}^{(i)}$  for which the morphism
\[
\textstyle f=(f_1, \dots, f_5): \calMc_{0,5}\to \prod_{i=1}^5\calMc_{0,4}{}^{(i)}
\]
is $\Sf_5$-equivariant. The following is well known.

\begin{lemma}\label{lemma:crossratiomap}
The  map  $f=(f_1, \dots, f_5): \calMc_{0,5}\to \prod_{i=1}^5\calMc_{0,4}{}^{(i)}$ is injective.
\end{lemma}
\begin{proof}
On  $\calM_{0,5}$ this is clear, since  given an ordered $5$-element subset $(x_1, \dots, x_5)$ of a smooth rational curve $C$, there is a 
unique affine coordinate  $z$ on $C$ such that $(z(x_1), z(x_2), z(x_3))=(0, \infty, 1)$,  and then for $i=4,5$ we have $z(x_i)=f_i(C;x_1,x_2, x_3, x_4)$. Since this even
allows $x_4=x_5$, we have in fact injectivity away from the zero dimensional strata. A typical zero dimensional stratum is represented in a Hilbert-Mumford  stable manner by a $5$-term sequence in $\{1,2,3\}$ with at exactly two repetitions, e.g.,  $(1,2, 3, 1,2)$ and then the value of  
$f_i$ on it is computed by removing the $i$th term and stipulating that of the four remaining two items two that are not repeated are made equal
(and are then renumbered in an order preserving manner such that all terms lie in $\{1,2\}$). (In our example the value of  $f_1$ is  
$(2,1,1,2)$, the value of $f_2,f_3, f_4$  is $(1,2,1,2)$, and the value of $f_5$ is $(1,2,2,1)$.) Elementary combinatorics shows that $f$ is here injective, too. 
\end{proof}

Recall  that 
$\calMc_{0,4}{}^{(i)}$ is  naturally a $3$-pointed smooth rational curve  and that the $\Sf_5$-stabilizer of the factor $\calMc_{0,4}{}^{(i)}$, $\Sf_{5,i}$ 
(the permutation group of $\{1, \dots , 5\}\ssm \{i\}$), acts on $\calMc_{0,4}{}^{(i)}$ through the full permutation group of these three 
points. Since  $\Sf_{5,i}\bs \calMc_{0,4}{}^{(i)}$ is independent of any numbering, it is canonically isomorphic to $\Sf_{4}\bs \calMc_{0,4}$, and hence with $\Pb^1$ (such an isomorphism can be given by the $j$-invariant of  the double cover of $\Pb^1$ ramifying in a prescribed $4$-element subset). So the 
$\Sf_5$-orbit space of $\prod_{i=1}^5\calMc_{0,4}{}^{(i)}$ is then  identified with $\Sym^5(\Pb^1)$. Since we can regard $\Sym^5(\Pb^1)$ as the linear system of degree  $5$ on $\Pb^1$, it is a projective space of dimension $5$. 
Let 
\begin{eqnarray}\label{mapPhi}
\textstyle \Phi:\Sf_5\bs\calMc_{0,5}\to \Sf_5\bs \prod_{i=1}^5\calMc_{0,4}{}^{(i)}\cong \Sym^5(\Pb^1)
\end{eqnarray}
be obtained from $f$ by passing to the $\Sf_5$-orbit spaces and denote by 
\[\tilde\Phi : \calMc_{0,5}\to \Sym^5(\Pb^1)\] its precomposite with the quotient morphism.
It follows from Lemma \ref{lemma:crossratiomap} that $\Phi$ is injective and Proposition \ref{prop:cone} tells us that $\Sf_5\bs\calMc_{0,5}$ 
naturally embeds  in a $5$-dimensional projective space as a quasi-cone.

\begin{theorem}\label{thm:}
The morphism $\tilde\Phi^*$ takes the hyperplane class of $\Sym^5(\Pb^1)$ to $-12K_S$, and $\Phi$ is projectively equivalent to  
the projective embedding found in Proposition \ref{prop:cone}. In particular, $\tilde\Phi(C_\infty)$ is a line in $\Sym^5(\Pb^1)$.
\end{theorem}
\begin{proof}
The hyperplane class of $\Sym^5(\Pb^1)$ pulled back to $(\Pb^1)^5$ is the sum of the degree 1 classes coming from the factors. The map 
\[\calMc_{0,4}\to \Sf_{4}\bs \calMc_{0,4}\cong \Pb^1\] is of degree $6$. The sum of the two special conics that appear as fibers of  
$f_5: \calMc_{0,5}\to \calMc_{0,4}$ represent the pull-back of a degree $2$ class on $\calMc_{0,4}$ and so the pull-back of the  degree 1 class
on  $ \Pb^1$ is $3$ times the sum of these two special conics. Hence the hyperplane class pulled back to  $\calMc_{0,5}$ is equal to 
$3(C_c+C'_c)=-12 K_S$. The degree of the image of $\Phi$ is then computed as 
\[(-12K_S)^2/\#(\Sf_5)=12^2\cdot 5/120=6.\] This is also the degree
of the image of $\Sf_5\bs\calMc_{0,5}$ in $\Pb^5$ embedded as a quasi-cone. One verifies in  straightforward manner that $\Cf_6\bs \Pb^2$ is simply-connected and so $\Pic(\Cf_6\bs \Pb^2)\cong H^2(\Cf_6\bs \Pb^2; \Z)$ is infinite cyclic. It follows that the 
two morphisms must be projectively equivalent.
\end{proof}

\begin{remark}\label{rem:clebsch} Let $V(5)$ be the space of binary quintic forms. It is classically known (and a fact that can be found in any text book on invariant  theory) that the subalgebra of $\SL_2$-invariants in the polynomial algebra $\C [V(5)]$ is generated by the Clebsch invariants $I_4$, $I_8$, $I_{12}$, $I_{18}$ of degrees  indicated in the subscripts (the last invariant is skew, i.e. it is not a $\GL_2$-invariant).  Their explicit form can be found in Salmon's 
book \cite{Salmon}. There is one basic relation among these invariants, which is of the form
\[
I_{18}^2 = P(I_4,I_8,I_{12}).
\] 
Since the subspace spanned by  elements of degrees $4k$ is the  subalgebra  freely generated by the first three invariants, we see that the GIT-quotient $\SL_2\bs\!\!\bs\Pb(V(5))$ is 
isomorphic to the weighted projective plane $\Pb(1,2,3)$. So in view of the discussion above, we obtain
\beq\label{binarymoduli}
\Sf_5\backslash S=\Sf_5\backslash\calMc_{0,5} \cong \Pb(1,2,3).
\eeq
With the help of exact sequence \ref{eq:exactsequence}  we can  identify $H^0(S,\omega_S^{-k})^{\Sf_5}$ 
with $\C[V(5)]_{2k}$. In other words, the zero set (or, more precisely, closed subscheme) of any invariant of degree $2k$ pulls back to a curve on $S$  defined by a $\Sf_5$-invariant section of  $\omega_S^{-k}$. It then makes sense  to ask what  the curves on $S$ are corresponding to the basic invariants $I_4$, $I_8$, $I_{12}$, $I_{18}$, as these correspond to $\Sf_5$-invariant sections of $\omega_S^{-2}$, $\omega_S^{-4}$, $\omega_S^{-6}$, $\omega_S^{-9}$ respectively. We see for example right away  that $I_4$ corresponds to our Wiman curve. 

The discriminant $\Delta$ of a binary quintic is a polynomial in coefficients of degree 8, and as such it must be a linear combination of $I_4^2$ and $I_8$. In fact, it is known that $\Delta = I_4^2-128I_8$. It is also a square of an element $D\in \C[V(5)]_4$ and so $D$ must correspond to the curve $C_\infty$. 
It is clear that any $\SL_2$-invariant in $\C[V(5)]_8$ is a linear combination of $\Delta=D^2$ and $I_4^2$, in other words, has the form $(aD+bI_4)(aD-bI_4)$
for some $(a, b)\not=(0,0)$.
So this defines on $S$  a reducible $\Sf_5$-invariant divisor  which is the sum of two $\Af_5$-invariant divisors. In other words, it  represents the union of two $\Sf_5$-conjugate members of the Wiman-Edge pencil.

The invariant $I_{18}$, being the, up to proportionality unique, invariant of degree $18$ must represent a unique curve   on $S$  defined by a 
$\Sf_5$-invariant section of $\omega_S^{-9}$. According to Clebsch \cite{Clebsch}, p.\ 298, the invariant $I_{18}$ vanishes on a binary form when its four zeros are invariant with respect to an involution of $\Pb^1$ and the remaining zero is a fixed point of this involution. This is precisely the locus parametrized  by 
$\Cf_2^{\ev}$ and according to Remark \ref{rem:irrcompA2locus} indeed defined by a section of $\omega_S^{-9}$.
\end{remark}

\end{document}